\documentclass[12pt]{article}
\usepackage{amsmath}
\usepackage{graphicx}
\usepackage{enumerate}
\usepackage{natbib}
\usepackage{subfigure}
\usepackage{url} 
\usepackage{multirow}
\usepackage{amsmath}
\usepackage{commath}
\usepackage{amsthm}
\usepackage{amssymb}
\usepackage{setspace,mathrsfs}
\usepackage{algorithm}
\usepackage{algorithmic}
\usepackage{url,amsmath,amsfonts,amssymb}
\usepackage{caption,algorithm,graphicx}
\usepackage{listings,enumerate,color,relsize,multirow,rotating,algorithm}
\usepackage{booktabs}
\usepackage{bbm}

\newcommand{\blind}{1}

\definecolor{purple}{RGB}{250,000,180}

\usepackage[percent]{overpic}
\usepackage{xcolor}

\definecolor{RawBlue}{HTML}{4682B4}
\definecolor{PCOrange}{HTML}{FF8C00}
\definecolor{RawBlue}{HTML}{4682B4}
\definecolor{PCOrange}{HTML}{FF8C00}

\RequirePackage[colorlinks,
            linkcolor=black,
            anchorcolor=black,
            citecolor=black,
            urlcolor=black,
            ]{hyperref}

\def\ttop{^{\top}}

\newcommand{\op}{\mathrm{op}}

\newcommand{\Diag}{\mathrm{Diag}}

\usepackage{mathrsfs}

\def \R {\mathbb{R}}
\def \P {\mathbf{P}}
\def \E {\mathbf{E}}

\newcommand{\ts}{\textstyle}
\newcommand{\cov}{\textup{cov}}
\newcommand{\var}{\textup{var}}

\newcommand{\tr}{\textup{tr}}

\newtheorem{lemma}{Lemma}

\newtheorem{proposition}{Proposition}
\newtheorem{assume}{Assumption}

\newtheorem{theorem}{Theorem}

\begin{document}

\def\spacingset#1{\renewcommand{\baselinestretch}%
{#1}\small\normalsize} \spacingset{1.5}


\if1\blind
{
  \title{\Large\bf A Goodness-of-Fit Test for Independent Component Models in High Dimensions}
  \author{Mingshuo Liu, Siyao Wang, and Miles E. Lopes\\
    University of California, Davis}
\date{}
  \maketitle
} \fi

\if0\blind
{
  \begin{center}
      {\LARGE\bf A Hypothesis Test for Independent Component Models under the High-Dimensional Regime}
\end{center}
  \medskip
} \fi

\begin{abstract}
\singlespacing
Independent component (IC) models are a standard tool for representing multivariate data in statistics, signal processing, and machine learning. Despite the extensive use of IC models, much less attention has been given to goodness-of-fit tests for assessing their compatibility with data. We develop the \emph{first} goodness-of-fit test for IC models that is supported by a theoretical validity guarantee when the data dimension and sample size diverge proportionally. This is  made possible by the fact that the test does not rely on a pre-whitening step, which often limits the applicability of other goodness-of-fit tests in high dimensions. Our theoretical analysis is complemented with numerical experiments that demonstrate the test's size control and power under a range of conditions. In addition, we provide examples involving gene-expression data to illustrate that the test has potential for effective diagnostic use in practice.

\end{abstract}

\noindent%
{\it Keywords: independent component model, goodness-of-fit,  high-dimensional statistics} \spacingset{1}

\spacingset{1.2}

\section{Introduction}\label{sec:intro}
We say that a set of observations $X_1,\dots,X_n\in\R^p$ is  generated from an independent component (IC) model if 
\begin{align}
\label{eqn:ICmodel}
X_i = \Sigma^{1/2} Z_i
\end{align}
holds for $i=1,\dots,n$, where $Z_1,\dots,Z_n\in\R^p$ are unobserved random vectors that form the columns of a $p\times n$ matrix consisting of i.i.d.~entries that have zero mean and unit variance, and $\Sigma\in\R^{p\times p}$ is an unknown matrix that is non-random and positive definite. Such models are widely used throughout statistics, signal processing, and machine learning as a standard tool for representing multivariate data. For example, IC models of the form~\eqref{eqn:ICmodel} appear frequently in contexts such as regression, classification, hypothesis testing, principal components analysis, and more~\citep[e.g.][]{bai2010spectral,rubio2012performance, yao2015sample,dobriban2018high,ledoit2020analytical,bodnar2022optimal, yao2023rates,bach2024high,pereira2024asymptotics}. Furthermore, such models are among the most well established in the random matrix theory literature, where they are sometimes called ``separable models''.
%
%
%
More broadly still, a large body of research based on extensions and variations of the model~\eqref{eqn:ICmodel} has been pursued in the field of independent component analysis (ICA)~\citep[e.g.][]{comon2010handbook, nordhausen2018independent,chen2021application, koldovsky2021dynamic, brendel2023unifying, asl2024spike, ishigami2024basic,auddy2025large}.

 In comparison to the extensive use of IC models, much less attention has been given to the goodness-of-fit problem, which seeks to test the hypothesis that data were generated from such a model. In particular, goodness-of-fit tests are valuable because they offer practitioners a systematic way to assess the reliability of model-based conclusions. However, the small handful of existing goodness-of-fit tests for IC models are not intended to handle high-dimensional data, and this is a serious issue because the applications of IC models frequently arise from such data. In response to this challenge, the key contribution of our work is a new goodness-of-fit test for IC models that is the \emph{first} to be supported by a theoretical validity guarantee in high-dimensional settings, including those with $p>n$.\\[-0.2cm]

\noindent {\textbf{Related work.}} Up to now, only a few papers have developed goodness-of-fit tests for IC models similar to~\eqref{eqn:ICmodel}~\citep{matteson2017independent,hallin2024consistent,schkoda2025goodness}. 
The first two of these rely on an estimate $\hat M$ of the matrix $\Sigma^{-1/2}$ in order to compute the ``pre-whitened'' vectors $\hat Z_i=\hat{M}X_i$, $i=1,\dots,n$, which are treated as proxies for $Z_1,\dots,Z_n$. After this is done, classical strategies for testing independence are then applied to the entries of $\hat Z_1,\dots,\hat Z_n$. 
On one hand, the tests in~\citep{matteson2017independent,hallin2024consistent} have the favorable property of being able to handle more general versions of the model~\eqref{eqn:ICmodel} that allow $\Sigma^{1/2}$ to be replaced with any invertible matrix. But on the other hand, they have an important limitation, which is that they \emph{need to accurately estimate the inverse of a $p\times p$ matrix}. Indeed, it is well known that this inverse matrix estimation problem is difficult even when the dimension $p$ is moderately large. Meanwhile,~\cite{schkoda2025goodness} consider a setting where $\Sigma^{1/2}$ is replaced with a matrix of coefficients associated with a structural equation model, and they develop an approach based elegant algebraic constraints satisfied by cumulant tensors. However, this approach relies on the asymptotic normality of certain matrices of cumulant estimates whose size increases with $p$, and methods for constructing such matrices in our context are not addressed. 
For these reasons, the three papers~\citep{matteson2017independent,hallin2024consistent,schkoda2025goodness} focus on low-dimensional settings, such that the theoretical results use classical asymptotics with $p$ held fixed as $n\to\infty$, and all empirical examples involve $p\leq 20$. 

In addition to IC models, there has been a flurry of recent interest in goodness-of-fit testing for several related models, especially in high dimensions. For instance, methods have been developed in the last few years for testing Gaussian models~\citep{chen2023normality, bing2025high,cui2026atm}, elliptical models~\citep{ tang2024nonparametric,wang2025testing}, graphical models~\citep{lin2025goodness,le2022testing}, and structural equation models~\citep{schultheiss2023,schultheiss2024,schkoda2025goodness}. Echoing the challenges of testing IC models discussed earlier, a prominent theme in this line of research is to avoid the drawbacks of pre-whitening, and several works have shown that a viable way to do this is by testing moment-based constraints~\citep{schultheiss2024,bing2025high,schkoda2025goodness,wang2025testing,cui2026atm}.

As outlined below, our proposed test for IC models also pursues a moment-based strategy and avoids pre-whitening, which is a key reason why the test is well-suited to high-dimensional settings.
Another point to emphasize is that our work targets moment constraints that are specific to IC models, and hence, it differs in essential ways from moment-based approaches for other models.\\[-0.2cm]

\noindent {\textbf{Overview of proposed method.}} At a conceptual level, our proposed method exploits the fourth-order moment structure of IC models by using an exact formula for the variance of quadratic forms. Namely, if $X_1\in\R^p$ is an observation generated from the IC model~\eqref{eqn:ICmodel}, and the component variable $Z_{11}$ has a finite fourth moment, then for any fixed symmetric matrix $A\in\R^{p\times p}$, we have
\begin{equation}
    \var\big( X_1\ttop A X_1\big) = 2\|\Sigma^{1/2} A \Sigma^{1/2}\|_F^2 + (\E(Z_{11}^4)-3) \sum_{j=1}^p (e_j\ttop \Sigma^{1/2} A \Sigma^{1/2} e_j)^2,
\end{equation}
where $\|\cdot\|_F$ is the Frobenius norm, and $e_j$ is the $j$th standard basis vector (see Lemma~\ref{lemma:covZMZ} in Appendix~\ref{sec:app,background}). The importance of this formula is that any choice of $A$ can serve as a ``probe'' to extract a constraint on an IC model. In essence, our method selects $p+1$ choices of such probe matrices and then aggregates them into a ``global'' constraint that is amenable to testing. The particular choices of the probe matrices are $A=I$, and $A=e_je_j\ttop$ for $j=1,\dots,p$, and in Lemma~\ref{lemma:covZMZ} they are combined to derive the following constraint equation
\begin{equation}\label{eqn:global}
      \frac{\var(\|X_1\|_2^2)-2\|\Sigma\|_F^2}{\sum_{j=1}^p\Sigma_{jj}^2}
  =\frac{\E(\|X_1\|_4^4)-3\sum_{j=1}^p\Sigma_{jj}^2}{\|\Sigma^{1/2}\|_4^4},
\end{equation}
where $\|\cdot\|_q$ is the entrywise $\ell_q$ norm on vectors and matrices for $q\geq 1$. Next, we develop separate estimates for the two sides of this constraint equation, and then use their difference, denoted $\hat \Delta_n$, as a test statistic that rejects the null hypothesis of an IC model when $|\hat\Delta_n|$ is sufficiently large.
Notably, because the equation~\eqref{eqn:global} is derived from $p+1$ model constraints, the statistic $\hat \Delta_n$ benefits from being sensitive to an increasing number of possible constraint violations as $p\to\infty$.\\[-0.2cm]

\noindent {\textbf{Methodological and theoretical challenges.}} To convert the previous intuition into a formal testing procedure, it is necessary to resolve the following challenges---all of which will be done in a high-dimensional setting where the limit $p/n\to\gamma$ holds for some constant $\gamma>0$ as $n\to\infty$. 
\begin{itemize}
\item Among all of the moment parameters appearing in~\eqref{eqn:global}, the estimation of $\|\Sigma^{1/2}\|_4^4$ is particularly non-trivial. This is because  the matrix square root function and the norm $\|\cdot\|_4$ lend themselves to conflicting types of matrix representations (i.e.~spectral and entrywise). Nevertheless, this estimation problem can be successfully handled using a bespoke variant of the linear shrinkage covariance estimators popularized by~\cite{ledoit2004well}, and we prove that the estimate is ratio-consistent for $\|\Sigma^{1/2}\|_4^4$ in Lemma~\ref{lem:l4normconsistency}.
\item Although the difference statistic $\hat \Delta_n$ has a clear interpretation as a measure of misfit, this statistic is not adequate for testing by itself, because its variance is unknown. In addition, it will turn out that the estimates used to construct $\hat \Delta_n$ are high-dimensional fourth degree polynomials. Consequently, deriving an asymptotic formula and a consistent estimate for $\var(\hat \Delta_n)$ are substantial undertakings.

\item After constructing an estimate $\hat\sigma_n^2$ for the variance of $\hat \Delta_n$, the last major technical challenge lies in proving a \emph{high-dimensional central limit theorem} for the standardized statistic  $\hat \Delta_n/\hat\sigma_n$ under the null hypothesis of an IC model (Theorem~\ref{thm:main}). This is made possible by the fact that $\hat \Delta_n$ can be approximated with a U statistic. However, it is crucial to note that because we work in a setting where $p/n\to\gamma>0$ as \smash{$n\to\infty$,} standard central limit theorems for U statistics are not applicable, and our proof is built from the ground up with precise moment calculations for high-dimensional quadratic forms.
\end{itemize}

\noindent {\textbf{Empirical performance.}} To complement the theoretical analysis of our method, Section~\ref{sec:simu} presents numerical results under null and alternative hypotheses across a range of high-dimensional conditions. These results show that the proposed test maintains accurate level control and is able to reliably detect non-IC models. Additionally, in Section~\ref{sec:realdata}, we present results from several GTEx gene-expression datasets, which illustrate that the test has the potential to be an effective diagnostic tool for practitioners.\\[-0.2cm]

\noindent\textbf{Notation.} For any $q\geq 1$, the $\ell_q$ norm of a real vector $v\in\R^p$ is $\|v\|_q=(\sum_{j=1}^p |v_j|^q)^{1/q}$, and for a real $p\times p$ matrix $A$, it is $\|A\|_q=(\sum_{1\leq i,j\leq p}|A_{ij}|^q)^{1/q}$.  When $A$ is symmetric, its sorted eigenvalues are denoted by $\lambda_{\max}(A)=\lambda_1(A) \geq  \cdots \geq \lambda_p(A)=\lambda_{\min}(A)$, its operator norm is given by $\|A\|_{\textup{op}}=\max_{1\leq j\leq p}|\lambda_j(A)|$, and its nuclear norm is given by $\|A\|_*=\sum_{j=1}^p|\lambda_j(A)|$. The effective rank of $A$ is defined as ${\tt{r}}(A)=\|A\|_*/\|A\|_{\textup{op}}$ if $A$ is non-zero, and ${\tt{r}}(0)=0$.\label{effrank} The square matrix that matches $A$ on the diagonal and is zero elsewhere is denoted by $\textup{Diag}(A)$. For two sequences of real numbers $\{a_n\}$ and $\{b_n\}$, we write $a_n=\mathcal{O}(b_n)$, if there exists a constant $C>0$, not depending on $n$, such that $|a_n| \leq C|b_n|$ holds for all large $n$. We also write $a_n\lesssim b_n$ as an equivalent notation for $a_n=\mathcal{O}(b_n)$, and if both of the relations $a_n \lesssim b_n$ and $b_n \lesssim a_n$ hold, then we write $a_n \asymp b_n$. If $a_n/b_n\to 0$ as $n\to\infty$, then we write $a_n=o(b_n)$. Lastly, when dealing with sequences of random variables, convergence in distribution is denoted by $\xrightarrow{\mathcal{L}}$.

\section{Method}
\label{sec:method}
\noindent {\textbf{Preliminaries.}}
When estimating the two sides of the constraint equation~\eqref{eqn:global}, we use separate halves of the data. This ensures that the statistic $\hat \Delta_n$ is a difference of two independent estimates, which will ultimately allow for a test with a tractable limiting null distribution. Likewise, to simplify notation, we always assume that the sample size $n$ is even. Also, to streamline the definition of many analogous estimates, we use $\hat{(\cdot)}$ to refer to estimates composed from  $\{X_{1},\dots,X_{n}\}$, and we use  $\check{(\cdot)}$ and $\tilde{(\cdot)}$  to respectively refer to the corresponding estimates composed from $\{X_1,\dots,X_{n/2}\}$ and $\{X_{n/2+1},\dots,X_{n}\}$. For example, we write $\hat\Sigma=\frac{1}{n}\sum_{i=1}^n X_iX_i\ttop$, $\check\Sigma=\frac{1}{n/2}\sum_{i\leq n/2} X_iX_i\ttop$, and $\tilde\Sigma=\frac{1}{n/2}\sum_{i>n/2} X_iX_i\ttop$.
 As one further preparation, if $f$ is a generic real-valued function on $\R^p$, we use $\hat{\E}(f(X_1))$ and $\hat{\var}(f(X_1))$ to respectively denote the sample mean and sample variance of the values $\{f(X_1),\dots,f(X_n)\}$.\\

\noindent{\textbf{Testing procedure.}} With the previous conventions in hand, the  difference statistic $\hat \Delta_n$ is defined as
\begin{equation}\label{eqn:hatdelta}
    \hat \Delta_n = 
    \frac{\check{\var}(\|X_1\|_2^2)-2\big(\|\check\Sigma\|_F^2-\frac{1}{n/2}\tr(\check\Sigma)^2\big)}{\sum_{j=1}^p\check\Sigma_{jj}^2}
  - 
  \frac{\tilde\E(\|X_1\|_4^4)-3\sum_{j=1}^p\tilde\Sigma_{jj}^2}{\|\tilde{\mathfrak{S}}^{1/2}\|_4^4},
\end{equation}
where $\tilde{\mathfrak{S}}$ is a shrinkage covariance estimator that will be introduced shortly. Whereas most of the parameters in the constraint equation~\eqref{eqn:global} can be estimated well using empirical counterparts, this is not the case for $\|\Sigma\|_F^2$, because $\|\check\Sigma\|_F^2$ has a substantial bias in high dimensions. For this reason, we correct the bias by using $\|\check\Sigma\|_F^2-\frac{1}{n/2}\tr(\check\Sigma)^2$ in equation~\eqref{eqn:hatdelta}.

The shrinkage covariance estimator is defined by
\begin{equation} \label{eq:defS}
\tilde{\mathfrak{S}}
  =\tilde s\, \tilde\Sigma+(1-\tilde s )\,\textup{Diag}(\tilde\Sigma),
\end{equation}
where the shrinkage parameter $\tilde s\in[0,1]$ is selected according to
\begin{equation} \label{eq:defs}
    1-\tilde s^2=\min\!\bigg\{
  \frac{\frac{2}{n}\tr(\tilde\Sigma)^2}
       {\sum_{i\ne j}\tilde\Sigma_{ij}^2},\,1
  \bigg\}.
\end{equation}
This value of $\tilde s$ is chosen so that $\|\tilde{\mathfrak{S}}\|_F^2$ matches the bias-corrected estimate $\|\tilde\Sigma\|_F^2-\frac{1}{n/2}\tr(\tilde\Sigma)^2$ in the typical case that $0<\tilde s<1$.

To complete the description of our testing procedure, it remains to provide an estimate for the variance of $\hat \Delta_n$, and specify the rule for rejecting the null hypothesis of an IC model. The estimate of $\var(\hat \Delta_n)$ is defined by
\begin{equation}\label{eqn:hatsigmadef}
    \hat\sigma_n^2= \frac{4(\hat{\var}(\|X_1\|_2^2))^2}{n(\sum_{j=1}^p \hat \Sigma_{jj}^2)^2} + \frac{ 2\hat \var(\|X_1\|_4^4)}{n\|\hat{\mathfrak{S}}^{1/2}\|_4^8},
\end{equation}
and its ratio-consistency for $\var(\hat \Delta_n)$ is established in Section~\ref{sec:sigma,cons}. Based on this estimate, our main result in Theorem~\ref{thm:main} shows that as $n\to\infty$, the limit $\hat \Delta_n/\hat\sigma_n\xrightarrow{\mathcal{L}} N(0,1)$ holds under the null hypothesis. Thus, for a given nominal level $\alpha\in(0,1)$, our goodness-of-fit test rejects whenever the event
\begin{equation}\label{eqn:test}
    |\hat \Delta_n|> \hat\sigma_n z_{1-\alpha/2}
\end{equation}
is observed, where $z_{1-\alpha/2}$ is the $(1-\alpha/2)$-quantile for the $N(0,1)$ distribution.

\section{Theory}
\label{sec:theory}
We analyze the proposed test in a standard asymptotic framework where the dimension $p$ and the data-generating distribution are allowed to vary with $n$.
In particular, this means that we deal with a sequence of goodness-of-fit testing problems that are implicitly indexed by $n$. The corresponding sequence of null hypotheses is defined by the conditions in Assumption~\ref{A:model} below. As a matter of notation for stating these conditions, recall that the effective rank ${\tt{r}}(A)$ of a symmetric matrix $A$ is defined at the end of Section~\ref{sec:intro}.

\begin{assume}
\label{A:model} 
~\\[-0.7cm]
\begin{itemize}
    \item[(a)] For each $i=1,\dots,n$, the $i^{th}$ observation has the form $X_i=\Sigma^{1/2}Z_i$, where $\Sigma\in\R^{p\times p}$ is a non-random positive definite matrix, and $Z_1,\dots,Z_n\in\R^p$ are the columns of the upper-left $p\times n$ block of a doubly infinite array of i.i.d.~random variables, such that $\E(Z_{11})=0$, $1=\E(Z_{11}^2)<\E(Z_{11}^4)$ and $\E(|Z_{11}|^{8+\delta}) \lesssim 1$ for some $\delta>0$ that is fixed with respect to $n$.
    \item[(b)] There is a constant $\gamma>0$ such that $p/n\to\gamma$ as $n\to\infty$.
    \item[(c)] The matrix $\Sigma$ satisfies $1\lesssim \lambda_{\min}(\Sigma)\leq\lambda_{\max}(\Sigma)\lesssim 1$, and ${\tt{r}}(\Sigma-I)=o(\sqrt p)$. 
\end{itemize}
\end{assume}

The following theorem is our main theoretical result, which implies that if the null hypothesis of an IC model holds, then the proposed test defined by~\eqref{eqn:test} has a rejection rate that asymptotically matches the nominal level. The proof is given in Appendix~\ref{app:proofmain}.

\begin{theorem}
\label{thm:main}
If Assumption \ref{A:model} holds, then as $n \rightarrow \infty$, 
\begin{align*}
\hat \Delta_n/\hat\sigma_n\,\xrightarrow{\mathcal{L}} N(0,1).
\end{align*}
\end{theorem}

\noindent {\textbf{Remarks.}}  To comment on Assumption~\ref{A:model}, the $(8+\delta)$-moment condition in part (a) makes it possible to establish the consistency of the variance estimate $\hat\sigma_n^2$, which involves eighth-degree polynomial functions of the data.  Part (c) arises from technical aspects of proving asymptotic normality of $\hat\Delta_n$. Notably, this part does not restrict the eigenvectors of $\Sigma$ at all, because ${\tt{r}}(\Sigma-I)$ can be expressed as a function of $(\lambda_1(\Sigma)-1,\dots,\lambda_p(\Sigma)-1)$.
 To provide some interpretation for the condition ${\tt{r}}(\Sigma-I)=o(\sqrt p)$, it may be viewed as a much weaker version of the condition that $\Sigma$ is a ``spiked covariance matrix''.  The simplest example of a spiked covariance matrix has the form $\Sigma=I+vv\ttop$ with $v\in\R^p$, which is to say that $\Sigma$ is a rank-1 perturbation of the identity. By contrast, part (c) allows for far more general matrices of the form $\Sigma=I+A$, where $A\in \R^{p\times p}$ may be \emph{full rank} and need not even be positive semidefinite. Note too that the possibility of ${\tt{r}}(A)\to\infty$ as $p\to\infty$ is allowed. For context, it is also worth emphasizing that conditions similar to part (c) commonly appear elsewhere in the literature on high-dimensional IC models~\citep[e.g.][]{yao2015sample, johnstone2018pca}. Lastly, the experiments in Section~\ref{sec:simu} illustrate that the test can perform well in settings that are even broader than Assumption~\ref{A:model}.

\section{Numerical results}
\label{sec:simu}

In this section, we evaluate the empirical performance of the proposed test in a suite of high-dimensional settings with simulated data. Section~\ref{sec:simu,separable} looks at level control, while Section~\ref{sec:ellipalter} demonstrates the power of our test in detecting departures from IC models. In all simulation settings, under the null or alternative, we compute the empirical rejection rates over 500 Monte Carlo trials using a nominal level of $\alpha=5\%$, and we fix $n=400$ while allowing $p\in\{100,400,600\}$. The code for implementing the test is available at ~\url{github.com/mingshuostat/ic-model-test}.

\vspace{0.5em}
\subsection{Level control}
\label{sec:simu,separable}

Here, we generate data under the null hypothesis of an IC model, so that the observations $X_i=\Sigma^{1/2} Z_i$, $i=1,\dots,n$, are i.i.d., with each vector $Z_i$ having centered and standardized i.i.d.~entries. We consider 36 settings corresponding to the  three aforementioned values for $p$, as well as four different distributions for $Z_{11}$, and three different structures for $\Sigma$.\\

\noindent\textbf{Covariance structures.} The three choices for $\Sigma$ are as follows.
\begin{enumerate}[(I)]
    \item  $\Sigma = I$.
    \item  $\Sigma_{ij}=0.3^{|i-j|}$.
    \item The eigenvalues of $\Sigma$ satisfy $\lambda_j(\Sigma)=j^{-1/4}$ for $j=1,\dots,p$, and the eigenvectors are drawn uniformly at random from the set of $p\times p$ orthogonal matrices.
\end{enumerate}

\noindent The covariance matrix (III) serves to illustrate the performance of our test when the theoretical conditions in Assumption~\ref{A:model} are violated. Namely, the smallest eigenvalue  $\lambda_{\min}(\Sigma)=p^{-1/4}$ is not bounded away from $0$ as $p\to\infty$, and also, ${\tt{r}}(\Sigma-I)$ is not $o(p^{1/2})$ because ${\tt{r}}(\Sigma-I)=\frac{1}{1-p^{-1/4}}\sum_{j=1}^p (1-j^{-1/4})\asymp p$. \\

\noindent\textbf{Component distributions.}
We generate $Z_{11}$ in four ways by standardizing the following distributions to satisfy  $\E(Z_{11})=0$ and $\mathrm{var}(Z_{11})=1$.

\begin{enumerate}[(1)]
    \item $t_{15}$,
    \item $\mathrm{Beta}(2,5)$,
    \item Laplace(0,1),
    \item Uniform(-1,1).
\end{enumerate}

\noindent\textbf{Results.}
Table~\ref{tab:separablesize} reports the empirical size (rejection rate) of the proposed test under all of the settings of the null hypothesis just described. In all cases, the test comes within about 2\% of the nominal level.

\begin{table}[H]
\centering
\caption{Empirical size of proposed test with nominal level $\alpha = 5\%$, $n = 400$.}
\label{tab:separablesize}
\begin{tabular}{c|cccc|cccc|cccc}
\hline\hline
& \multicolumn{4}{c|}{$p = 100$} & \multicolumn{4}{c|}{$p = 400$} & \multicolumn{4}{c}{$p = 600$} \\
& (1) & (2) & (3) & (4) & (1) & (2) & (3) & (4) & (1) & (2) & (3) & (4) \\
\hline
(I)   & .050 & .043 & .039 & .050 & .042 & .042 & .050 & .068 & .044 & .036 & .043 & .065 \\
(II)  & .044 & .044 & .034 & .058 & .038 & .055 & .043 & .045 & .039 & .046 & .044 & .054 \\
(III) & .047 & .047 & .039 & .049 & .040 & .032 & .039 & .063 & .046 & .039 & .053 & .066 \\
\hline
\end{tabular}
\end{table}

\subsection{Power}
\label{sec:ellipalter}
We now study the power of the proposed test in detecting violations of IC models. In these experiments, we do not make comparisons with the power of the goodness-of-fit tests for IC models in~\citep{matteson2017independent,hallin2024consistent}, because they are not intended for high-dimensional settings. (Recall from the introduction that all of the empirical examples in these works deal with $p\leq 20$.) \\

\noindent{\textbf{Alternative distributions.} }To construct a family of distributions with varying degrees of separation from an IC model, we use an interpolation defined through a parameter $h\in[0,1]$, which starts at an IC model with $h=0$, and ends at an elliptical model with $h=1$. Specifically, for a given choice of $h$, we consider i.i.d.~data $X_1,\dots,X_n\in\R^p$ such that
\begin{equation*}\label{eqn:alt}
X_1 = \Sigma^{1/2}(\sqrt{1-h}Z_1+\sqrt h \eta_1 U_1),
\end{equation*}
where $Z_1\in\R^p$ has i.i.d.~$N(0,1)$ entries, and $(\eta_1,U_1)\in\R^{p+1}$ is a vector independent of $Z_1$, such that $U_1\in\R^p$ is uniformly distributed on the unit sphere, and $\eta_1\geq 0$ is a random variable independent of $U_1$ that is normalized as $\E(\eta_1^2)=p$. This interpolation ensures that $\E(X_1)=0$ and $\E(X_1X_1\ttop)=\Sigma$ for every choice of $h$. We also allow $\Sigma$ to vary among the three structures (I), (II), (III) given in the previous subsection, and we consider the  following three choices for the distribution $\eta_1$ (defined through its square),
\begin{enumerate}[(1)]
\setcounter{enumi}{4}
\item \quad $\eta_1^2 \sim \mathrm{Poisson}(p)$
\item \quad $\eta_1^2 \sim (p+4)\,\mathrm{Beta}(p/2,2)$
\item \quad $\eta_1^2 \sim (p-20)\,F(p,p/10)$.
\end{enumerate}
For each setting of $\Sigma$ and $\eta_1$, we select values of $h$ having the form $h=cg$, where $g\in \{0.3,0.4,\dots,0.8\}$, and $c$ is a multiplier that is chosen to yield comparable rejection rates across the three choices of $\eta_1$. Namely, we take $c=1.2$ in case (5), $c=1$ in case (6), and $c=0.3$ in case (7).\\

\noindent{\textbf{Results.}} 
Tables~\ref{tab:underellip_h_uniform}-\ref{tab:underellip_h_ellip3} separately report empirical rejection rates in settings corresponding to each of the three covariance structures (I)-(III). The results are intuitive in that the rejection rates increase monotonically with the separation from the IC model (measured by $g$), and also, in that the rejection rates approach 1 when $g$ is sufficiently large. In addition, it is notable that for a fixed choice of $g$ and $\eta_1$, the power of the test is broadly stable with respect to variations in the choices of $p$ and $\Sigma$.

\begin{table}[!htbp]
\centering
\caption{Empirical rejection rates for covariance structure (I) with nominal level $\alpha = 5\%$.}
\label{tab:underellip_h_uniform}
\begin{tabular}{c|ccc|ccc|ccc}
\hline\hline
\multirow{2}{*}{$g$} & \multicolumn{3}{c|}{$p = 100$} & \multicolumn{3}{c|}{$p = 400$} & \multicolumn{3}{c}{$p = 600$} \\
& (5) & (6) & (7) & (5) & (6) & (7) & (5) & (6) & (7) \\
\hline
0.3 & .079 & .116 & .203 & .081 & .148 & .125 & .105 & .144 & .120 \\
0.4 & .204 & .349 & .417 & .231 & .419 & .288 & .242 & .425 & .261 \\
0.5 & .448 & .766 & .677 & .497 & .850 & .543 & .522 & .831 & .497 \\
0.6 & .826 & .990 & .852 & .880 & .998 & .778 & .889 & .997 & .735 \\
0.7 & .993 & 1 & .949 & .997 & 1 & .945 & .996 & 1 & .908 \\
0.8 & 1 & 1 & .985 & 1 & 1 & .987 & 1 & 1 & .986 \\
\hline
\end{tabular}
\end{table}

\begin{table}[!htbp]
\centering
\caption{Empirical rejection rates for covariance structure (II) with nominal level $\alpha = 5\%$.}
\label{tab:underellip_h_correlated}
\begin{tabular}{c|ccc|ccc|ccc}
\hline\hline
\multirow{2}{*}{$g$} & \multicolumn{3}{c|}{$p = 100$} & \multicolumn{3}{c|}{$p = 400$} & \multicolumn{3}{c}{$p = 600$} \\
& (5) & (6) & (7) & (5) & (6) & (7) & (5) & (6) & (7) \\
\hline
0.3 & .076 & .111 & .160 & .081 & .109 & .106 & .083 & .131 & .095 \\
0.4 & .165 & .281 & .346 & .158 & .304 & .207 & .170 & .307 & .200 \\
0.5 & .347 & .593 & .576 & .364 & .682 & .417 & .381 & .656 & .360 \\
0.6 & .652 & .917 & .762 & .706 & .966 & .675 & .718 & .959 & .622 \\
0.7 & .931 & .997 & .909 & .955 & 1 & .853 & .964 & 1 & .820 \\
0.8 & .998 & 1 & .957 & 1 & 1 & .957 & 1 & 1 & .944 \\
\hline
\end{tabular}
\end{table}

\begin{table}[!htbp]
\centering
\caption{Empirical rejection rates for covariance structure (III) with nominal level $\alpha = 5\%$.}
\label{tab:underellip_h_ellip3}
\begin{tabular}{c|ccc|ccc|ccc}
\hline\hline
\multirow{2}{*}{$g$} & \multicolumn{3}{c|}{$p = 100$} & \multicolumn{3}{c|}{$p = 400$} & \multicolumn{3}{c}{$p = 600$} \\
& (5) & (6) & (7) & (5) & (6) & (7) & (5) & (6) & (7) \\
\hline
0.3 & .073 & .107 & .185 & .094 & .133 & .116 & .092 & .146 & .097 \\
0.4 & .178 & .304 & .378 & .210 & .351 & .247 & .193 & .344 & .220 \\
0.5 & .398 & .695 & .624 & .446 & .754 & .493 & .459 & .748 & .432 \\
0.6 & .751 & .966 & .818 & .811 & .993 & .727 & .806 & .987 & .663 \\
0.7 & .980 & 1 & .930 & .982 & 1 & .899 & .987 & 1 & .860 \\
0.8 & 1 & 1 & .975 & 1 & 1 & .976 & 1 & 1 & .967 \\
\hline
\end{tabular}
\end{table}

\section{Illustration with gene-expression data}
\label{sec:realdata}

Gene expression data provide a natural illustration for the performance of our goodness-of-fit test, because such data are inherently high-dimensional, and also, because independent component analysis has become a popular tool in the genomic literature~\citep{wang2021independent, urzua2021improving, sastry2021independent, anglada2022robustica}. This section presents results based on RNA-seq data derived from the Genotype-Tissue Expression (GTEx) project, which has played a prominent role in modern research on the links between genetic variation and human health~\citep{lonsdale2013genotype}.

We use four datasets consisting of RNA-seq expression profiles corresponding to four tissue types (testis, colon–sigmoid, stomach, and pancreas), which were processed by~\cite{khunsriraksakul2022integrating} and are available at~\url{github.com/mingshuostat/ic-model-test}. The respective values of $(n,p)$ for the datasets are (272, 35007), (265, 24483), (260, 24290) and (244, 22615). From each of the associated $n\times p$ data matrices, we randomly generated 100 smaller matrices of size $n\times d$ with $d\in\{10,20,\dots,300\}$ by sampling columns without replacement.  Thus, a total of $4\times 30\times 100=12000$ data matrices were generated. Next, we centered the rows in every such data matrix with the corresponding sample mean vector, and applied the proposed test to compute a p-value. This resulted in 100 p-values for each pair $(n,d)$ and tissue type. 

In Figure~\ref{fig:real,compare}, the blue curves display the median of the 100 p-values as a function of $d$ for each tissue type. All the blue curves exhibit a similar decreasing pattern. Namely, when $d=10$, the median p-values are between 0.2 and 0.6, and then they steadily decrease until they are close to 0 for $d\geq 70$, suggesting that the test is well-suited to high-dimensional settings. The fact that the p-values stabilize for large $d$ is also consistent with our simulation results in the previous section, which show a stable relationship between power and dimension when the dimension is at least 100.

To make the goodness-of-fit problem more challenging, we repeated the previous p-value computations using ``noisy'' versions of the data. Specifically, we computed the singular value decomposition of each of the $n\times d$ matrices described above, and then replaced the largest $d/5$ of the singular values with 0. This can be interpreted as removing the top 20\% of the principal components, so that the modified data may be viewed informally as ``noisy residuals''.  Similar approaches based on residuals are also used in~\citep{matteson2017independent} and~\citep{hallin2024consistent} for the purpose of constructing more difficult goodness-of-fit problems from natural data.

Analogously to the original data, we obtained a set of 100 p-values for each pair $(n,d)$ and tissue type from the noisy data. 
 In Figure~\ref{fig:real,compare}, we use orange curves to display the noisy-data counterparts of the blue curves. For nearly all values of $d$, each orange curve exceeds its associated blue curve, which conforms with the idea that noisier data increase the difficulty of a goodness-of-fit problem. Apart from having larger values, the orange curves exhibit an overall similarity with the blue curves, in that they take relatively large values when $d$ is small and approach 0 when $d$ becomes large. As with the original data, these results for the noisy data provide a meaningful indication that the test can be powerful in high-dimensional settings.

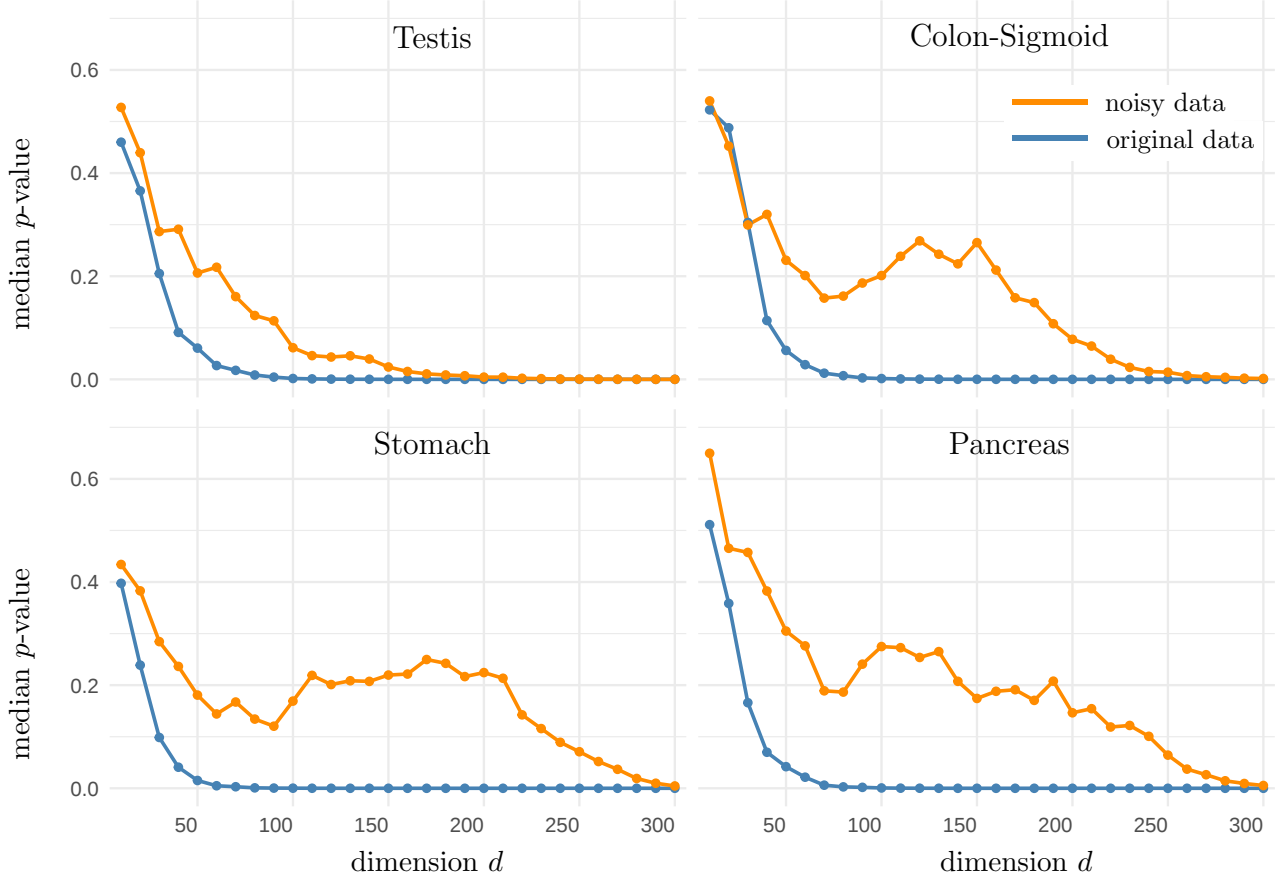
\begin{figure}[H]
\centering
\begin{overpic}[width=\linewidth]{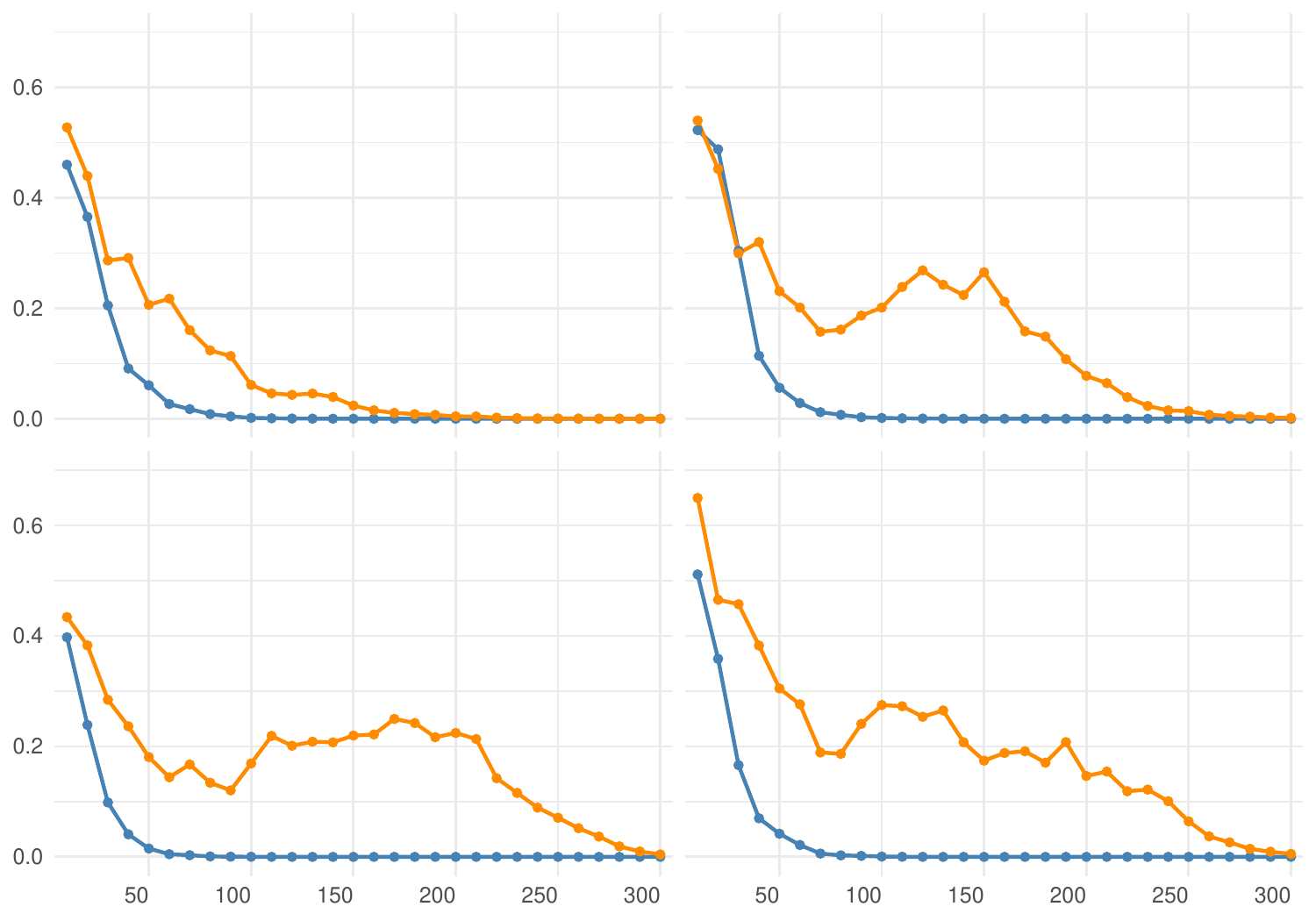}

\put(30,66){\makebox(0,0){ Testis}}
\put(77,66){\makebox(0,0){ Colon-Sigmoid}}
\put(30,33){\makebox(0,0){ Stomach}}
\put(77,33){\makebox(0,0){ Pancreas}}

\put(30,-1){\makebox(0,0){\small dimension $d$}}
\put(78,-1){\makebox(0,0){\small dimension $d$}}
\put(-3,50){\rotatebox{90}{\makebox(0,0){\small median $p$-value}}}
\put(-3,15){\rotatebox{90}{\makebox(0,0){\small median $p$-value}}}

\put(77,60){\colorbox{white}{\footnotesize\textcolor{PCOrange}{${}^{{}_{\rule{10mm}{2pt}}}$}\;\;noisy data}}
\put(77,57){\colorbox{white}{\footnotesize\textcolor{RawBlue}{${}^{{}_{\rule{10mm}{2pt}}}$}\; original data}}

\end{overpic}
\vspace{0.1cm}

\caption{
Median $p$-values versus gene-subset size $d\in\{10,20,\ldots,300\}$ for four tissues: Testis, Colon--Sigmoid, Stomach, Pancreas. 
}
\label{fig:real,compare}
\end{figure}

\section*{Acknowledgements}
We are grateful to Xin Bing, Derek Latremouille, 
David Matteson, Klaus Nordhausen and Lida Wang for helpful correspondence.

\bibliographystyle{chicago}

\bibliography{library}

\newpage
\appendix
\section*{\centering Supplementary Material\\ A Goodness-of-Fit Test for Independent Component Models in High Dimensions}

 Appendix \ref{app:proofmain} contains the proof of Theorem~\ref{thm:main}, which is organized into three subsections: Section~\ref{sec:kappa1,asym}  proves the asymptotic normality of a statistic that approximates $\hat\Delta_n$. Section~\ref{sec:kappa2,asym} establishes the negligibility of remainder terms. Section~\ref{sec:sigma,cons} proves the ratio consistency of the variance estimate $\hat\sigma_n^2$. Lastly, Appendix~\ref{sec:app,background} contains various background results.\\

\noindent\textbf{Notation.} Here we define some additional notation that is needed for the supplement and is not already defined in the notation paragraph of Section~\ref{sec:intro}. We use $\xrightarrow{\P}$ to denote convergence in probability. For sequences of random variables $\{V_n\}$ and $\{W_n\}$, the relation $W_n = o_{\P}(V_n)$ means that $W_n/V_n \xrightarrow{\P} 0$, and $W_n = \mathcal{O}_{\P}(V_n)$ means that for every fixed $\epsilon>0$, there exists a constant $C>0$ not depending on $n$, such that $\sup_{n\ge1}\P(|W_n/V_n|\geq C)\leq\epsilon$. For any fixed $q\geq 1$, the $L^q$ norm of a random variable $V$ is given by $\|V\|_{L^q}=(\E(|V|^q))^{1/q}$.

\section{Proof of Theorem \ref{thm:main}}
\label{app:proofmain}
Define the parameter
\begin{equation}\label{eqn:sigmandef}
\sigma_{n}^2
  =
  \frac{4\,\var^2(\|X_1\|_2^2)}{n\,(\sum_{j=1}^p \Sigma_{jj}^2)^2}.
\end{equation}
In Section~\ref{sec:sigma,cons}, we show that the limit
\begin{align*}
\frac{\hat \sigma_n^2}{\sigma_{n}^2} \overset{\P}{\rightarrow} 1
\end{align*}
holds as $n\to\infty$.
Thus, the proof of Theorem~\ref{thm:main} reduces to showing
\begin{align*}
\frac{\hat\Delta_n}{\sigma_{n}} \xrightarrow{\mathcal{L}} N(0,1)
\end{align*}
as $n\to\infty$. To this end, consider the decomposition of $\hat\Delta_n$ given by
\begin{equation*}
    \hat\Delta_n = D_n+\epsilon_{n,1} - \epsilon_{n,2}-\epsilon_{n,3},
\end{equation*}
where we define the random variables 
\begin{equation}\label{eqn:decompdefs}
    \begin{split}
        D_n&=\frac{\check{\var}(\|X_1\|_2^2)-\var(\|X_1\|_2^2)-2(\|\check\Sigma\|_F^2-\|\Sigma\|_F^2-\frac{1}{n/2}\tr(\check\Sigma)^2)}{\sum_{j=1}^p\check\Sigma_{jj}^2}\\[0.4cm]
        \epsilon_{n,1}&=\frac{\big(\var(\|X_1\|_2^2)-2\|\Sigma\|_F^2\big)\sum_{j=1}^p (\Sigma_{jj}^2 - \check \Sigma_{jj}^2)}{(\sum_{j=1}^p \Sigma_{jj}^2)(\sum_{j=1}^p \check\Sigma_{jj}^2)}  \\[0.4cm]
    \epsilon_{n,2} &=\ts\frac{1}{\|\Sigma^{1/2}\|_4^4}\Big(\tilde\E(\|X_1\|_4^4)-\E(\|X_1\|_4^4)- 3\sum_{j=1}^p (\tilde\Sigma_{jj}^2-\Sigma_{jj}^2)\Big)\\[0.4cm]
    \epsilon_{n,3} &=  \Big(\tilde\E(\|X_1\|_4^4)-3 \sum_{j=1}^p \tilde \Sigma_{jj}^2\Big) \Big(\ts \frac{1}{\| \tilde{\mathfrak{S}}^{1/2}\|_4^4 } - \frac{1}{\|  \Sigma^{1/2}\|_4^4 } \Big).
    \end{split}
\end{equation}
In Proposition~\ref{prop:kappa1CLT} of Section~\ref{sec:kappa1,asym}, we show that  $\frac{D_n}{\sigma_n}\xrightarrow{\mathcal{L}} N(0,1)$ as $n\to\infty$.
Finally, in Section~\ref{sec:kappa2,asym}, we show that
    $\frac{\epsilon_{n,k}}{\sigma_n} = o_{\P}(1)$
holds for $k=1,2,3$ in Lemmas~\ref{lem:eps1},~\ref{lem:eps2},~\ref{lem:eps3} respectively, which completes the proof.

\qed

\subsection{Asymptotic normality of $D_n$}
\label{sec:kappa1,asym}

\begin{proposition}
\label{prop:kappa1CLT}
If Assumption \ref{A:model} holds, then as $n \rightarrow \infty$,
\begin{align*}
\frac{D_n}{ \sigma_{n}} \xrightarrow{\mathcal{L}} N(0,1).
\end{align*}    
\end{proposition}

\noindent \textit{Proof.} 
Define the statistic
\begin{align}\label{eqn:Udef}
& U_n = \check{\var}(\|X_1\|_2^2)-2\big(\|\check\Sigma\|_F^2-\ts\frac{1}{n/2}\tr(\check\Sigma)^2\big),
\end{align}
so that  $D_n/\sigma_n$ can be expressed as

\begin{equation}
    \frac{D_n}{\sigma_n}=\frac{\sqrt n\big(U_n - \var(\|X_1\|_2^2)+2\|\Sigma\|_F^2\big)}{ 2\var(\|X_1\|_2^2)}\cdot \frac{\sum_{j=1}^p\Sigma_{jj}^2}{\sum_{j=1}^p \check \Sigma_{jj}^2}. 
\end{equation}
Regarding the second factor on the right side, it follows from Lemma~\ref{lem:sumhatsjj2}  that
\begin{equation}
    \frac{\sum_{j=1}^p\Sigma_{jj}^2}{\sum_{j=1}^p \check \Sigma_{jj}^2} = 1+o_{\P}(1). 
\end{equation}
The proof is completed by Lemma~\ref{lem: U normality}, which shows that
\begin{equation}
    \frac{\sqrt n\big(U_n - \var(\|X_1\|_2^2)+2\|\Sigma\|_F^2\big)}{ 2\var(\|X_1\|_2^2)}\xrightarrow{\mathcal{L}} N(0,1)
\end{equation}
as $n\to\infty$.

\qed

\begin{lemma}
\label{lem:sumhatsjj2}
If Assumption \ref{A:model} holds, then as $n\to\infty$, 
\begin{align*}
\sum_{j=1}^p \hat\Sigma_{jj}^2 - \Sigma_{jj}^2  = \mathcal{O}_{\P}(1) \ \ \text{  \ \ and } \qquad \frac{\sum_{j=1}^p \hat\Sigma_{jj}^2}{\sum_{j=1}^p \Sigma_{jj}^2} \xrightarrow{\P} 1.
\end{align*} 
In addition, the same statements hold when $\hat\Sigma$ is replaced by $\tilde\Sigma$ or $\check\Sigma$.
\end{lemma}

\noindent \textit{Proof.} 
Consider the algebraic identity
\begin{align}\label{eqn:2terms}
\sum_{j=1}^p \hat\Sigma_{jj}^2 -  \Sigma_{jj}^2 &=
 \|\textup{Diag}(\hat\Sigma)-\textup{Diag}(\Sigma)\|_F^2 \ + \ 2\,\tr\Big( \textup{Diag}(\Sigma)[\textup{Diag}(\hat\Sigma)-\textup{Diag}(\Sigma)]\Big).
\end{align}
Because $\sum_{j=1}^p \Sigma_{jj}^2\gtrsim p$ holds under Assumption~\ref{A:model}, the proof will be completed if we can show both terms on the right side of~\eqref{eqn:2terms} are $\mathcal{O}_{\P}(1)$. For the first term, we only need to show that its expectation is $\mathcal{O}(1)$, since it is a non-negative random variable. Notice that
\begin{align*}
\E \big(\|\textup{Diag}(\hat\Sigma) - \textup{Diag}(\Sigma)\|_{F}^2\big) 
= \sum_{j=1}^p \var(\hat\Sigma_{jj}) \,\lesssim\,  \frac{1}{n}\sum_{j=1}^p\E(X_{1j}^4),
\end{align*}
where the inequality is due to the classical formula  for the variance of the sample variance (\citealp[Page 13]{lee2019u}). To handle $\E(X_{1j}^4)$, we use Lemma \ref{lemma:covZMZ} to get
\begin{equation}
\label{eqn:EX1j4}
\begin{aligned}
\E(X_{1j}^4) &= 3\Sigma_{jj}^2 + (\E(Z_{11}^4)-3) \sum_{l=1}^p (\Sigma^{1/2})_{lj}^4\\
&\lesssim \Sigma_{jj}^2 + \Big(\sum_{l=1}^p (\Sigma^{1/2})_{lj}^2\Big)^2\\
&=2\Sigma_{jj}^2\\
&\lesssim 1
\end{aligned}
\end{equation}
where the last step uses $\|\Sigma\|_{\textup{op}} \lesssim 1$ under Assumption \ref{A:model}. Hence, $\E (\|\textup{Diag}(\hat\Sigma) - \textup{Diag}(\Sigma)\|_{F}^2) \lesssim 1$.

For the second term on the right side of~\eqref{eqn:2terms}, we have
\begin{align*}
\tr\Big( \textup{Diag}(\Sigma)[\textup{Diag}(\hat\Sigma)-\textup{Diag}(\Sigma)]\Big)=\frac{1}{n}\sum_{i=1}^n\Big(\sum_{j=1}^p \Sigma_{jj}(X_{ij}^2-\Sigma_{jj})\Big)=:\frac 1n\sum_{i=1}^n \zeta_i,
\end{align*}
where the random variables $\zeta_1,\dots,\zeta_n$ are i.i.d.~with mean 0. It suffices to show that $\frac{1}{n} \var(\zeta_1) \lesssim 1$, and we can calculate  this variance as 
\begin{align*}
\var(\zeta_1) = \sum_{j=1}^p \sum_{k=1}^p \Sigma_{jj} \Sigma_{kk} \cov(X_{1j}^2, X_{1k}^2),
\end{align*}
where
\begin{align*}
\cov(X_{1j}^2, X_{1k}^2) = 2\Sigma_{jk}^2 + ( \E(Z_{11}^4)-3) \sum_{l=1}^p (\Sigma^{1/2})_{lj}^2 (\Sigma^{1/2})_{lk}^2, 
\end{align*}
as implied by Lemma \ref{lemma:covZMZ}. 
It follows that
$$
\var(\zeta_1)\lesssim 
\|\Sigma\|_F^2+ \|(\Sigma^{1/2})^{\circ 2}\mathbf{1}\|_2^2,
$$
where $((\Sigma^{1/2})^{\circ 2})_{ij}:=((\Sigma^{1/2})_{ij})^2$, and $\mathbf{1}\in\R^p$ is the all-ones vector.
Moreover, the assumptions on $\Sigma$ imply that $\|\Sigma\|_F^2\lesssim p$, and it is known that $\|M^{\circ 2}\|_{\textup{op}}\leq\|M\|_{\textup{op}}^2$ holds for any real matrix $M$ \cite[Theorem 5.5.15]{horn1994topics}. Thus,
$$
\|(\Sigma^{1/2})^{\circ 2}\mathbf{1}\|_2^2\leq \big(\|\Sigma\|_{\textup{op}}\|\mathbf{1}\|_2\big)^2\lesssim p,
$$
which completes the proof.

\qed

\begin{lemma}
Let the statistic $U_n$ be as defined in~\eqref{eqn:Udef}. If Assumption \ref{A:model} holds, then as $n \rightarrow \infty$,
\begin{align*}
\frac{\sqrt{n}\big(U_n-\var(\|X_1\|_2^2)+2\|\Sigma\|_F^2\big)}{2\var(\|X_1\|_2^2)}\xrightarrow{\mathcal{L}} N(0,1),
\end{align*}
\label{lem: U normality}
\end{lemma}

\noindent \textit{Proof.}
Define the kernel function
\begin{equation}\label{eqn:hdef}
h(X_1, X_2) = \big(\ts\frac{1}{2} - \frac{2(n-2)}{n^2}\big) \big(\|X_1\|_2^2 - \|X_2\|_2^2\big)^2 - \frac{2(n-2)}{n} (X_1^{\top} X_2)^2.
\end{equation}
so that $U_n$ can be represented in the following form as a U statistic,
\begin{align}
\label{equa: U stat}
U_n = \frac{1}{\binom{n/2}{2}} \sum_{1 \leq i<j\leq n/2} h(X_i, X_j).
\end{align}
Next, define the function
\begin{equation}\label{eqn:h1def}
h_1(X_1) = \E(h(X_1, X_2) | X_1) - \E(h(X_1, X_2)),
\end{equation}
and let $L_n$ denote the H\'ajek projection of $U_n-\E(h(X_1, X_2))$, which satisfies
\begin{align}
L_n = \frac{2}{n/2} \sum_{i=1}^{n/2} h_1(X_i).
\label{eqn: hatU def}
\end{align}
Lemma \ref{lem: varhatUU} shows that $\var(L_n)/\var(U_n)\to 1$ as $n\to\infty$, and consequently, \cite[][Theorem 11.2]{van2000asymptotic} ensures
\begin{align}
\frac{U_n-\E(U_n)}{(\var(U_n))^{1/2}} = \frac{L_n}{(\var(L_n))^{1/2}} + o_{\P}(1).
\label{eqn: U normal}
\end{align}
The proof will be complete if we can establish two more items, which are
\begin{align}
\frac{L_n}{(\var(L_n))^{1/2}} \xrightarrow{\mathcal{L}}N(0,1)
\label{eqn: hatU normal}
\end{align}
and
\begin{align}
\frac{U_n-\E(U_n)}{\sqrt{\var(U_n)}} = \frac{\sqrt{n}(U_n-(\var(\|X_1\|_2^2)-2\|\Sigma\|_F^2))}{2\var(\|X_1\|_2^2)} +  o_{\P}(1). 
\label{eqn: std U decom}
\end{align}
The limit \eqref{eqn: hatU normal} is established in Lemma \ref{lem: hatU normal}, and so it remains to prove \eqref{eqn: std U decom}. Using Lemma \ref{lemma:covZMZ}, direct calculations give 
\begin{equation}\label{eqn: EhX1X2}
\begin{aligned}
\E(U_n)&=\E(h(X_1, X_2))\\
& = \Big(\ts\frac{1}{2} - \ts\frac{2(n-2)}{n^2}\Big) \cdot 2\cdot \var(\|X_1\|_2^2) - \frac{2(n-2)}{n} \E\big((X_1^{\top} X_2)^2\big) \\
& =  (\var(\|X_1\|_2^2) - 2\|\Sigma\|_F^2) + \mathcal{O}(n^{-1}) (\var(\|X_1\|_2^2) + \|\Sigma\|_F^2).
\end{aligned}
\end{equation}
Furthermore, Lemma \ref{lem: varhatUU} implies that
\begin{align}
\frac{\var(U_n)}{(4/n)\var^2(\|X_1\|_2^2)} \rightarrow 1,
\label{eqn: varU order}
\end{align}
as $n \rightarrow \infty$. Combining \eqref{eqn: EhX1X2} and \eqref{eqn: varU order}, as well as the fact that $ \|\Sigma\|_F^2\lesssim \var(\|X_1\|_2^2)$ under Assumption~\ref{A:model}, it follows that~\eqref{eqn: std U decom} holds.

\qed

\begin{lemma}
Let the statistics $U_n$ and $L_n$ be as defined in~\eqref{eqn:Udef} and~\eqref{eqn: hatU def} respectively. If Assumption \ref{A:model} holds, then as $n \rightarrow \infty$,
\begin{align}\label{eqn:varUvarhatU}
    \frac{\var(L_n)}{\var(U_n)} \rightarrow 1,
\end{align}
and
\begin{align}\label{eqn:varhatUformula}
\frac{\var(L_n)}{(4/n)\var^2(\|X_1\|_2^2)} \rightarrow 1.
\end{align}
\label{lem: varhatUU}
\end{lemma}

\noindent \textit{Proof.}
It is a classical fact~\citep[p.163]{van2000asymptotic} that the variances of $U_n$ and $L_n$ can be expressed as
\begin{align*}
& \var(U_n) = \ts\frac{2(n/2-2)}{\binom{n/2}{2}}\var(h_1(X_1)) + \ts\frac{1}{\binom{n/2}{2}} \var(h(X_1,X_2))\\
& \var(L_n) = \ts\frac{4}{n/2} \var(h_1(X_1)),
\end{align*}
where we recall that $h$ and $h_1$ are defined in~\eqref{eqn:hdef} and~\eqref{eqn:h1def} respectively.
Thus, the limit~\eqref{eqn:varUvarhatU} will follow if we can show
\begin{align}
\var(h(X_1,X_2)) = o(n\, \var(h_1(X_1))).
\label{eqn: psi2_n1}
\end{align}
We will establish this by showing $\var(h(X_1,X_2)) \lesssim \|\Sigma\|_F^4$ and $\var(h_1(X_1)) \asymp \|\Sigma\|_F^4$  below. For the upper bound, we have
\begin{align*}
\var(h(X_1,X_2)) 
& \ \lesssim \  \var\big((\|X_1\|_2^2 - \|X_2\|_2^2)^2\big) + \var\big((X_1^{\top} X_2)^2\big)\\
& \ \lesssim \  \E\big(\big(\|X_1\|_2^2 - \tr(\Sigma) \big)^4\big) +  \E\big((X_1^{\top} X_2)^4\big)\\ 
& \ \lesssim \|\Sigma\|_F^4,
\end{align*}
where Lemmas~\ref{lemma: EZ1MZ2^4} and~\ref{lem:Bai,quadra} have been used in the last step.

Now, we calculate the order of $\var(h_1(X_1))$. Recall that $h_1(X_1)=\E(h(X_1,X_2)|X_1)-\E(h(X_1,X_2))$, and note that the definition of $h$ gives
\begin{equation}\label{eqn: EhX1X2|X1}
\begin{aligned}
\E(h(X_1, X_2) | X_1) 
& = \big(\ts\frac{1}{2} - \frac{2(n-2)}{n^2}\big) \Big( \|X_1\|_2^4 -2 \tr(\Sigma) \|X_1\|_2^2  + \E(\|X_2\|_2^4) \Big) - \frac{2(n-2)}{n} X_1^{\top}\Sigma X_1.
\end{aligned}
\end{equation}
Due to the fact that variance is shift invariant, we may ignore $\E(h(X_1,X_2))$ and also replace $\E(\|X_1\|_2^4)$ with $\tr(\Sigma)^2$ to write the second factor in parentheses as $(\|X_1\|_2^2-\tr(\Sigma))^2$ when calculating $\var(h_1(X_1))$.  It follows that
\begin{equation}\label{eqn: psi1 decom}
\begin{split}
\var(h_1(X_1)) 
& =\var\big(\E\big(h(X_1, X_2)|X_1\big)\big) \\[0.1cm]
& = (\ts\frac{1}{4} + \mathcal{O}(\frac{1}{n}) ) \var\big( \big(\|X_1\|_2^2 -\tr(\Sigma)\big)^2\big) + (4+\mathcal{O}(\frac{1}{n})) \var(X_1^{\top} \Sigma X_1) \\[0.1cm]
& ~~~ - (2+\mathcal{O}(\ts\frac{1}{n})) \cov\big(\big( \|X_1\|_2^2 - \tr(\Sigma)\big)^2, X_1^{\top} \Sigma X_1\big).
\end{split}
\end{equation}
Lemma~\ref{lem:varnormsq} shows that 
\begin{equation}\label{eqn:varFrob}
    \begin{split}
        \var\Big(\big[\|X_1\|_2^2-\tr(\Sigma)\big]^2\Big) &\, = \, (2+o(1))\var^2(\|X_1\|_2^2)\\
    &\,  \asymp  \, \|\Sigma\|_F^4,
    \end{split}
\end{equation}
    where the second step follows from Lemma~\ref{lemma:covZMZ}  and the condition $\E(Z_{11}^4)>\E(Z_{11}^2)=1$ in Assumption~\ref{A:model}.
Meanwhile, Lemma~\ref{lemma:covZMZ} and Assumption~\ref{A:model} imply
\begin{equation}
    \var(X_1\ttop \Sigma X_1)\lesssim \tr(\Sigma^4)=o(\|\Sigma\|_F^4),
\end{equation}
and so the term $\var(X_1\ttop \Sigma X_1)$ is negligible in~\eqref{eqn: psi1 decom}. Likewise, the Cauchy-Schwarz inequality implies that the covariance term in~\eqref{eqn: psi1 decom} is also negligible, and hence
\begin{equation}
    \label{eqn:varh1order}
    \var(h_1(X_1))=(\ts\frac{1}{2}+o(1))\var^2(\|X_1\|_2^2)\asymp \|\Sigma\|_F^4.
\end{equation}
This verifies~\eqref{eqn: psi2_n1} and completes the proof of the first limit~\eqref{eqn:varUvarhatU} in the statement of the result. Finally, the second limit~\eqref{eqn:varhatUformula} follows from 
$$\var(L_n)= \frac{8}{n}\var(h_1(X_1))=\frac{4}{n}(1+o(1))\var^2(\|X_1\|_2^2),$$ 
which completes the proof.

\qed

\begin{lemma}\label{lem:varnormsq}
    If Assumption \ref{A:model}  holds, then as $n\to\infty$,
    \begin{equation*}
        \var\Big(\big[\|X_1\|_2^2-\tr(\Sigma)\big]^2\Big)=(2+o(1))\var^2(\|X_1\|_2^2).
    \end{equation*}
\end{lemma}

\noindent \textit{Proof.} Lemma \ref{lem: quadra normal} shows that if $\|\Sigma\|_F/\|\Sigma\|_{\textup{op}}\to\infty$ as $n\to\infty$, then
\begin{align}\label{eqn:quadclt}
\frac{\|X_1\|_2^2 - \tr(\Sigma)}{\sqrt{\var(\|X_1\|_2^2)}}\xrightarrow{\mathcal{L}} N(0,1).
\end{align}
The ratio $\|\Sigma\|_F/\|\Sigma\|_{\textup{op}}$ diverges under Assumption~\ref{A:model}, because we have $\|\Sigma\|_{\textup{op}}\asymp 1$ and $\|\Sigma\|_F^2 \gtrsim p$,
so the limit~\eqref{eqn:quadclt} holds.

Next, Lemmas~\ref{lemma:covZMZ} and~\ref{lem:Bai,quadra} can be used to obtain the moment bound
\begin{equation}\label{eqn:unifintbound}
    \E\,\bigg(\bigg| \frac{\|X_1\|_2^2 - \tr(\Sigma)}{\sqrt{\var(\|X_1\|_2^2)}}\bigg|^{4+\frac{\delta}{2}} \bigg)\lesssim 1,
\end{equation}
 with $\delta$ being as in Assumption~\ref{A:model}. Combining~\eqref{eqn:quadclt} and~\eqref{eqn:unifintbound}, uniform integrability implies that as $n\to\infty$, we have
\begin{equation*}
       \E\,\bigg(\bigg| \frac{\|X_1\|_2^2 - \tr(\Sigma)}{\sqrt{\var(\|X_1\|_2^2)}}\bigg|^4\bigg)
       \to \frac{1}{\sqrt{2\pi}}\displaystyle\int_{-\infty}^{\infty} z^4 e^{-z^2/2} dz= 3.
\end{equation*}
Thus,
\begin{align*}
    \var\big(\big(\|X_1\|_2^2-\tr(\Sigma)\big)^2\big) 
    & = \E\big((\|X_1\|_2^2-\tr(\Sigma))^4\big)-\var^2(\|X_1\|_2^2)\\
    & =(3+o(1))\var^2(\|X_1\|_2^2)-\var^2(\|X_1\|_2^2)
\end{align*}
which completes the proof.

\qed

\begin{lemma}
If Assumption \ref{A:model}  holds, then as $n\to\infty$,
$$
\frac{L_n}{(\var(L_n))^{1/2}} \xrightarrow{\mathcal{L}} N(0,1).
$$
\label{lem: hatU normal}
\end{lemma}

\noindent \textit{Proof.} 
Recall that $L_n = \frac{2}{n/2} \sum_{i=1}^{n/2} h_1(X_i)$, and let $\delta>0$ be as in Assumption~\ref{A:model}. It suffices to verify that the following Lyapunov-type condition holds as $n\to\infty$,
\begin{align}
\frac{\E(|h_1(X_1)|^{2+\delta/4})}{n^{\delta/8} (\var(h_1(X_1)))^{1+\delta/8}} \rightarrow 0.
\label{cond: lyapunaov}
\end{align}
Note that this condition implies the Lindeberg central limit theorem in the setting of triangular arrays \cite[Proposition 2.27]{van2000asymptotic}, which allows for $p$ to diverge with $n$ asymptotically. From the definitions of $h_1(X_1)$ and $h(X_1, X_2)$ in~\eqref{eqn:h1def} and~\eqref{eqn:hdef}, we have
\begin{equation}
\begin{split}
\E\big(|h_1(X_1)|^{2+\delta/4}\big) 
& \ \lesssim \ \E\big(|h(X_1,X_2)|^{2+\delta/4}\big)\\
& \  \lesssim  \ \E\big(\big|\|X_1\|_2^2 - \tr(\Sigma)\big|^{4+\delta/2}\big) \ + \  \E\big(|X_1^{\top} X_2|^{4+\delta/2}\big)\\
& \ \lesssim \  \|\Sigma\|_F^{4+\delta/2},
\end{split}
\end{equation}
where Lemmas~\ref{lemma: EZ1MZ2^4} and~\ref{lem:Bai,quadra} have been used in the last step.
Meanwhile, we know from~\eqref{eqn:varh1order} that $(\var(h_1(X_1)))^{1+\delta/8}  \asymp \|\Sigma\|_F^{4+\delta/2}$, which implies the desired condition~\eqref{cond: lyapunaov}.

\qed

\begin{lemma}
\label{lemma: EZ1MZ2^4}
 If Assumption \ref{A:model}  holds with the value of $\delta>0$ stated there, and $2\leq q\leq 4+\delta/2$, then 
\begin{align*}
\E(|X_1^{\top} X_2|^q) \lesssim \|\Sigma\|_F^q.
\end{align*}
\end{lemma}

\noindent \textit{Proof.}
Note that $X_1^{\top} X_2 = Z_1^{\top} \Sigma Z_2$, and so conditionally on $Z_1$, the quantity $X_1^{\top} X_2$ may be regarded as a linear combination of centered independent random variables. Applying Rosenthal's inequality (Lemma \ref{lem:Rosenthal}) conditionally on $Z_1$, the following bound holds almost surely for some absolute constant $c>0$,
\begin{equation}\label{eqn:condZ1}
\begin{split}
  \E\big(|Z_1\ttop \Sigma Z_2|^q \big| Z_1\big)
& \leq (cq)^q \,\max \bigg\{ \E^{q/2}\big((Z_1\ttop \Sigma Z_2)^2|Z_1\big), \sum_{j=1}^p \E\big(|(Z_1\ttop \Sigma e_j) Z_{2j}|^q|Z_1\big) \bigg\} \\
& =  (cq)^q \,\max\bigg\{ \|\Sigma Z_1\|_2^q\, , \, \E(|Z_{21}|^q) \|\Sigma Z_1\|_q^q\bigg\}\\[0.2cm]
& \leq (c q)^q\, (1+\E(|Z_{21}|^q))(Z_1\ttop \Sigma^2 Z_1)^{q/2}.
\end{split} 
\end{equation}
Moreover, using Lemma~\ref{lem:Bai,quadra} and the triangle inequality for the $L^{q/2}$ norm, we have
\begin{equation}
\begin{split}
    \big\|Z_1^{\top} \Sigma^2 Z_1\big\|_{L^{q/2}} & \ \leq \big\|Z_1^{\top} \Sigma^2 Z_1 - \|\Sigma\|_F^2\big\|_{L^{q/2}}+ \|\Sigma\|_F^2\\
    &\lesssim \|\Sigma^2\|_F+ \big(\tr(\Sigma^q))^{\frac{1}{q/2}}+\|\Sigma\|_F^2\\
    &\lesssim \|\Sigma\|_F^2.
    \end{split}
\end{equation}
Taking an expectation over $Z_1$ in~\eqref{eqn:condZ1} and applying the previous bound completes the proof.
\qed

\subsection{Negligibility of remainders}
\label{sec:kappa2,asym}
Recall that the remainder variables $\epsilon_{n,1}$, $\epsilon_{n,2}$, and $\epsilon_{n,3}$ are defined in~\eqref{eqn:decompdefs}, and the parameter $\sigma_n$ is defined in \eqref{eqn:sigmandef}.

\begin{lemma}\label{lem:eps1}
If Assumption \ref{A:model} holds, then as $n \rightarrow \infty$
\begin{equation*}
    \frac{\epsilon_{n,1}}{\sigma_n}=o_{\P}(1).
\end{equation*}
\end{lemma}
\noindent \textit{Proof.}
The definitions of $\epsilon_{n,1}$ and $\sigma_n$ imply that
\begin{equation*}
 \frac{\epsilon_{n,1}}{\sigma_n}=\frac{\sqrt n \big(\var(\|X_1\|_2^2)-2\|\Sigma\|_F^2\big)}{2\,\var(\|X_1\|_2^2)}\cdot \frac{ \sum_{j=1}^p(\Sigma_{jj}^2 - \check \Sigma_{jj}^2)}{\sum_{j=1}^p \Sigma_{jj}^2}\cdot\frac{\sum_{j=1}^p \Sigma_{jj}^2}{\sum_{j=1}^p \check\Sigma_{jj}^2}. 
 \end{equation*}
By Lemma~\ref{lemma:covZMZ} and the conditions in Assumption~\ref{A:model}, we have 
\begin{equation}
\label{eqn:varorder}
    \var(\|X_1\|_2^2)\asymp \|\Sigma\|_F^2
\end{equation}
and
\begin{equation*}
    \Big|\var(\|X_1\|_2^2)-2\|\Sigma\|_F^2\Big|\lesssim \|\Sigma\|_F^2
\end{equation*}
which imply
$$\frac{\sqrt n\big(\var(\|X_1\|_2^2)-2\|\Sigma\|_F^2\big)}{\,\var(\|X_1\|_2^2)} =\mathcal{O}(\sqrt n). $$
Next, since we have $\sum_{j=1}^p \Sigma_{jj}^2\gtrsim p$ under Assumption~\ref{A:model}, it follows from Lemma~\ref{lem:sumhatsjj2} that
\begin{equation*}
 \frac{\sum_{j=1}^p \Sigma_{jj}^2}{\sum_{j=1}^p \check \Sigma_{jj}^2} = \mathcal{O}_{\P}(1) \quad \text{and} \quad \frac{ \sum_{j=1}^p\Sigma_{jj}^2 - \check \Sigma_{jj}^2}{\sum_{j=1}^p \Sigma_{jj}^2} =\mathcal{O}_{\P}(\ts\frac{1}{p}).
\end{equation*}
Altogether, we conclude that $\frac{\epsilon_{n,1}}{\sigma_n}=\mathcal{O}_{\P}(\frac{\sqrt n}{p})=o_{\P}(1)$, as needed.

\qed

\begin{lemma}\label{lem:eps2}
If Assumption \ref{A:model} holds, then as $n \rightarrow \infty$
\begin{equation*}
    \frac{\epsilon_{n,2}}{\sigma_n}=o_{\P}(1).
\end{equation*}
\end{lemma}

\noindent \textit{Proof.} Recall that
\begin{equation*}
    \epsilon_{n,2} =\ts\frac{1}{\|\Sigma^{1/2}\|_4^4}\Big(\tilde\E(\|X_1\|_4^4)-\E(\|X_1\|_4^4)- 3\sum_{j=1}^p (\tilde\Sigma_{jj}^2-\Sigma_{jj}^2)\Big)
\end{equation*}
Lemmas~\ref{lem:sumhatsjj2} and~\ref{lem:centralXl4}  respectively show that
\begin{equation*}
\begin{split}
    \sum_{j=1}^p \tilde\Sigma_{jj}^2-\Sigma_{jj}^2 =\mathcal{O}_{\P}(1) \quad \text{and} \quad \sqrt{\var(\tilde\E(\|X_1\|_4^4))} = o(p^{1/4}).
\end{split}
\end{equation*}
Also, Lemma~\ref{lem:l4diagD} ensures  $\|\Sigma^{1/2}\|_4^4\gtrsim p$, and so
$\epsilon_{n,2}=o_{\P}(p^{-3/4})$.
To conclude, it follows from~\eqref{eqn:varorder} and the conditions in Assumption~\ref{A:model} that
\begin{equation}\label{eqn:sigmaorder}
    \sigma_n\asymp \frac{1}{\sqrt n}, 
\end{equation}
yielding the stated result.

\qed

\begin{lemma}
\label{lem:centralXl4}
If Assumption \ref{A:model} holds, then as $n\to\infty$,
\begin{align}\label{eqn:L2L4bound}
\var( \|X_1\|_4^4) = o(p^{3/2}).
\end{align}
\end{lemma}

\noindent\textit{Proof.}
We begin with some notation and preliminary observations. Let $A$ and $B$ denote the $p\times p$ matrices that satisfy 
\begin{align}\label{eqn:Deltadef}
\Sigma=I+A\quad \text{and} \quad
\Sigma^{1/2}=  I+B.
\end{align}
 The eigenvalues of $B$ can be expressed as $\lambda_j(B)=\sqrt{1+\lambda_j(A)}-1$, and so the bound $|\sqrt{1+x} - 1| \leq |x|$  for all $x \geq -1$ implies
\begin{align}
\label{eqn:deltajbound}
|\lambda_j(B)|\, \leq \, |\lambda_j(A)|
\end{align}
for all $j=1,\dots,p$.
Next, let $B_j\in\R^p$ denote the $j$th column of $B$ so that
$$X_{1j}= Z_{1j}+ B_j\ttop Z_1.$$
To establish the bound~\eqref{eqn:L2L4bound}, the above expression for $X_{1j}$ gives
\begin{align*}
\sqrt{\var( \|X_1\|_4^4)} &= \Big\|\sum_{j=1}^p X_{1j}^4-\E(X_{1j}^4)\Big\|_{L^2}\\
& \lesssim \sum_{k=0}^4 \bigg\| \sum_{j=1}^p Z_{1j}^{(4-k)}(B_j\ttop Z_1)^k - \E\big(Z_{1j}^{(4-k)}(B_j\ttop Z_1)^k\big)\bigg\|_{L^2}\\
&=: \sum_{k=0}^4 t_k.
\end{align*}

We now proceed to bound each of the terms $t_0,\dots,t_4$ in the last sum.
The quantity $t_0$ is simply the standard deviation of a sum of centered i.i.d.~random variables and so
$$t_0=\sqrt{ \var(Z_{11}^4)\, p }\lesssim \sqrt p.$$
Next, to bound $t_1$, we have
\begin{align*}
t_1&\lesssim \sum_{j=1}^p  \big\|Z_{1j}^3 (B_j\ttop Z_1)\big\|_{L^2}\\
&\leq \|Z_{11}\|_{L^{8}}^3 \sum_{j=1}^p \|B_j\ttop Z_1\|_{L^{8}},
\end{align*}
where H\"older's inequality with conjugate exponents $(4/3,4)$ has been used in the last step.
Since $B_j\ttop Z_1$ is a sum of centered independent random variables, the norm $\|B_j\ttop Z_1\|_{L^{8}}$ can be bounded with Rosenthal's inequality (Lemma \ref{lem:Rosenthal}), 
\begin{equation}
\begin{aligned}
\| B_j\ttop Z_1 \|_{L^{8}} & \lesssim \max\Big\{ \|B_j\ttop Z_1\|_{L^2}, \big(\sum_{k=1}^p \|B_{jk} Z_{1k}\|_{L^{8}}^{8}\big)^{1/8}  \Big\}\\
& \lesssim \max\big\{ \|B_{j}\|_2, \|B_j\|_8\big\} \\
& \leq \|B_j\|_2.
\end{aligned}    
\label{eqn: rosendeltaj}
\end{equation}
Therefore, the eigenvalue bound in~\eqref{eqn:deltajbound} leads to
\begin{align*}
t_1 &\, \lesssim \, \sum_{j=1}^p \|B_j\|_2 \, \leq \,  \sqrt{p} \|B\|_F \, \lesssim \, \sqrt{p}  \|A\|_F.
\end{align*}
Similarly, with regard to $t_k$ for $k=2,3,4$, the bound $\max_{1\leq j\leq p}\|B_j\|_2\leq \|B\|_{\textup{op}}\lesssim 1$ implies
\begin{equation*}
\begin{aligned}
   t_k
   &\lesssim \|Z_{11}\|_{L^{8}}^{(4-k)} \sum_{j=1}^p \|B_j\ttop Z_1\|_{L^{8}}^k \\
    &\lesssim \sum_{j=1}^p \|B_j\|_2^k\\
    &\lesssim \sum_{j=1}^p \|B_j\|_2^2\\
    &=\|B\|_F^2\\[0.1cm]
    &\leq \|A\|_F^2,
\end{aligned}
\label{eqn:Tk}
\end{equation*}
where we have used the eigenvalue bound~\eqref{eqn:deltajbound} in the last step.
Combining the previous calculations and applying Assumption~\ref{A:model}, we conclude
\begin{align*}
\sqrt{\var(\|X_1\|_4^4)}
& \ \lesssim \ \sqrt{p} + \sqrt{p}\, \|A\|_F + \|A\|_F^2 \\
& \leq \ \sqrt{p} + \sqrt{p}\, (\|A\|_* \|A\|_\textup{op})^{1/2} + \|A\|_* \|A\|_\textup{op}  \\
& \ = \  o(p^{3/4}).
\end{align*}

\qed

\begin{lemma}\label{lem:eps3}
If Assumption \ref{A:model} holds, then as $n\to\infty$,

\begin{equation*}
    \frac{\epsilon_{n,3}}{\sigma_n}=o_{\P}(1).
\end{equation*}

\end{lemma}
\noindent \textit{Proof.} It was noted in~\eqref{eqn:sigmaorder} that $\sigma_n\asymp 1/\sqrt n$, and so it suffices to show  $\epsilon_{n,3}=o_{\P}(1/\sqrt n)$. Recall that
\begin{equation*}
    \epsilon_{n,3}=\Big(\tilde\E(\|X_1\|_4^4)-3 \sum_{j=1}^p \tilde \Sigma_{jj}^2\Big) \Big( \ts \frac{1}{\| \tilde{\mathfrak{S}}^{1/2}\|_4^4 } - \frac{1}{\|  \Sigma^{1/2}\|_4^4 } \Big).
\end{equation*}
By Lemma~\ref{lem:sumhatsjj2}, we have $\sum_{j=1}^p \tilde\Sigma_{jj}^2=\mathcal{O}_{\P}(p)$,  and from the bound~\eqref{eqn:EX1j4} we have $\E(\tilde\E(\|X_1\|_4^4))=\sum_{j=1}^p \E(X_{1j}^4)\lesssim p$, which gives
$$\tilde\E(\|X_1\|_4^4)-3 \sum_{j=1}^p \tilde \Sigma_{jj}^2=\mathcal{O}_{\P}(p).$$ 
Also, due to Lemma~\ref{lem:l4diagD}, we know that 
\begin{equation}\label{eqn:l4normlower}
\|\Sigma^{1/2}\|_4^4\gtrsim p.
\end{equation}
 Consequently, using the basic identity $1/a-1/b= (b-a)/(a^2+a(b-a))$, the proof will be complete if we can show $    \|\tilde{\mathfrak{S}}^{1/2}\|_4^4-\|\Sigma^{1/2}\|_4^4 = o_{\P}(\sqrt n)$,
which is the content of the following lemma.\qed

\begin{lemma}\label{lem:l4normconsistency}
    If Assumption \ref{A:model}  holds, then as $n\to\infty$,
\begin{equation}\label{eqn:keydiff}
    \|\tilde{\mathfrak{S}}^{1/2}\|_4^4-\|\Sigma^{1/2}\|_4^4 = o_{\P}(\sqrt n),
\end{equation}    and in particular
$\|\tilde{\mathfrak{S}}^{1/2}\|_4^4/\|\Sigma^{1/2}\|_4^4\xrightarrow{\P}1.$

\end{lemma}
\proof The second limit follows from the first because of the bound~\eqref{eqn:l4normlower}.
We will establish~\eqref{eqn:keydiff} by separately considering the diagonal and off-diagonal terms in the left hand side. To handle the diagonal terms, Lemma~\ref{lem:l4diagD} shows that
\begin{equation*}
    \sum_{j=1}^p (\tilde{\mathfrak{S}}^{1/2})_{jj}^4-(\Sigma^{1/2})_{jj}^4 = o_{\P}(\sqrt n).
\end{equation*}
For the off-diagonal entries of $\tilde{\mathfrak{S}}^{1/2}$, Lemma~\ref{lem:offdiag4thpowers} shows that $\sum_{j\neq k}(\tilde{\mathfrak{S}}^{1/2})_{jk}^4=o_{\P}(n^{-1/2})$. Finally, to bound $ \sum_{j\neq k}(\Sigma^{1/2})_{jk}^4$,  recall that we use $A$ and $B$ to denote the matrices that satisfy $\Sigma=I+A$ and $\Sigma^{1/2}=I+B$. From~\eqref{eqn:deltajbound}, we know that $|\lambda_j(B)| \leq |\lambda_j(A)|$ for all $j=1,\dots,p$ and so
\begin{equation}
\begin{aligned}
\label{eqn:offSigmahalf4}
\sum_{j\neq k} (\Sigma^{1/2})_{jk}^4 &=  \sum_{j \neq k} B_{jk}^4\\
&\leq \|B\|_{\textup{op}}^2 \|B\|_F^2\\
&\lesssim  \|A\|_{\textup{op}}^3\|A\|_*\\[0.1cm]
&= o(\sqrt n),
\end{aligned}    
\end{equation}
where Assumption~\ref{A:model} has been used in the last step.

\qed

\begin{lemma}\label{lem:offdiag4thpowers}
If Assumption \ref{A:model} holds, then as $n\to\infty$,
    \begin{equation*}
        \sum_{j\neq k}(\tilde{\mathfrak{S}}^{1/2})^4_{jk} \,=\,o_{\P}(n^{-1/2}).
    \end{equation*}
\end{lemma}
\proof
Rearranging the definition of $\tilde{\mathfrak{S}}$ in \eqref{eq:defS} gives
\begin{align*}
\tilde{\mathfrak{S}} 
 = \textup{Diag}(\tilde \Sigma) + \tilde s (\tilde \Sigma - \textup{Diag}(\tilde \Sigma)).
\end{align*}
In Lemma~\ref{lem:1-t order}, it is shown that $\tilde s=o_{\P}(n^{-1/4})$, which motivates approximating the matrix square root of $\tilde{\mathfrak{S}}$ with that of $\textup{Diag}(\tilde\Sigma)$ via a matrix-valued Taylor expansion. Such an expansion requires $\tilde{\mathfrak{S}}$ and $\Diag(\tilde\Sigma)$ to be positive definite, and it follows from Lemmas~\ref{lem:1-t order} and~\ref{lem:phatSS} that this occurs with probability tending to 1 as $n\to\infty$. For this reason, the event that either of these matrices is singular will not affect the following analysis. 

To proceed, let $\tilde R\in\R^{p\times p}$ be a remainder matrix defined so that 
\begin{equation}
    \label{eq:defR}\tilde{\mathfrak{S}}^{1/2}= \textup{Diag}(\tilde\Sigma)^{1/2}+ \tilde s \int_0^{\infty} e^{-t \textup{Diag}(\tilde\Sigma)^{1/2}}(\tilde\Sigma-\textup{Diag}(\tilde \Sigma))e^{-t \textup{Diag}(\tilde\Sigma)^{1/2}}\mathrm{d}t+ \tilde R, 
\end{equation}
where the integral corresponds to the differential of the matrix square root function at $\textup{Diag}(\tilde\Sigma)$, acting upon the difference matrix $\tilde{\mathfrak{S}}-\Diag(\tilde\Sigma)=\tilde s (\tilde \Sigma - \textup{Diag}(\tilde \Sigma))$~\citep[][Theorem 1.1]{del2018taylor}. Because the expansion is done around a diagonal matrix, the integral can be computed explicitly, which leads to
\begin{align}\label{eqn:siqsqrtexpand}
(\tilde{\mathfrak{S}}^{1/2})_{jk} 
& =  \frac{\tilde s \tilde \Sigma_{jk}}{(\tilde \Sigma_{jj})^{1/2} + (\tilde \Sigma_{kk})^{1/2}} + \tilde R_{jk} \ \ \ \ \ \text{ for $j\neq k$}. 
\end{align}
 Furthermore, it is known from equation 9 in~\citep[]{del2018taylor} that the operator norm of the remainder satisfies the bound
\begin{equation*}
\begin{aligned}
\|\tilde R\|_{\textup{op}} & \leq \frac{ \tilde s^2 \big\| \tilde \Sigma - \textup{Diag}(\tilde \Sigma) \big\|_{\textup{op}}^2}{2(\lambda_{\min}(\textup{Diag}(\tilde \Sigma)))^{3/2}}  
\end{aligned}
\label{eqn:operatorR}
\end{equation*}
almost surely. Regarding the order of this bound, Lemma~\ref{lem:1-t order} shows that $\tilde s^2=o_{\P}(p^{-1/2})$, Lemma~\ref{lem:phatSS} shows that $\lambda_{\min}(\textup{Diag}(\tilde\Sigma))\geq c-o_{\P}(1)$ for a constant $c>0$ not depending on $n$, and Lemma \ref{lem:cov op bounds} implies $\| \tilde \Sigma - \textup{Diag}(\tilde \Sigma) \|_{\textup{op}}=\mathcal{O}_{\P}(1)$, leading to
\begin{align*}
\|\tilde R\|_{\textup{op}} = o_{\P}(p^{-1/2}).
\end{align*}
Similarly, when the lemmas just mentioned are applied to~\eqref{eqn:siqsqrtexpand} with the bound $\|\tilde R\|_4^4\leq \|\tilde R\|_{\textup{op}}^2\|\tilde R\|_F^2\leq p\|\tilde R\|_{\textup{op}}^4$, we obtain
\begin{align}
\sum_{j \neq k} (\tilde{\mathfrak{S}}^{1/2})_{j k}^4
& =  o_{\P}(\ts\frac{1}{p})\displaystyle\sum_{j\neq k} \tilde\Sigma_{jk}^4 + o_{\P}(\ts\frac{1}{p}).
\label{eqn:offSigmat4}
\end{align}
Next, we will show that the sum of off-diagonal terms on the right side is $o_{\P}(\sqrt n)$ by showing that its expectation is $o(\sqrt n)$. To see this, first observe that 
\begin{equation*}
    \sum_{j\neq k}\E(\tilde\Sigma_{jk}^4) \ \lesssim \ \sum_{j\neq k}\Big(\|\tilde\Sigma_{jk}-\Sigma_{jk}\|_{L^4}^4+ \Sigma_{jk}^4\Big).
\end{equation*}
It has been shown in~\eqref{eqn:offSigmahalf4} that $\sum_{j\neq k}\Sigma_{jk}^4=o(n^{1/2})$, and also, Rosenthal's inequality (Lemma \ref{lem:Rosenthal}) implies
\begin{align*}
\|\tilde\Sigma_{jk}-\Sigma_{jk}\|_{L^4}  & \, =\, 
\Big\|\ts\frac{1}{n/2} \sum_{i>\frac{n}{2}}\big( X_{ij} X_{ik} - \Sigma_{jk}\big) \Big\|_{L^4}\\
& \lesssim \max \bigg\{ n^{-1/2} \sqrt{\var(X_{1j} X_{1k})}, n^{-1}\Big(\sum_{i>\frac{n}{2}}^n \|X_{ij} X_{ik} - \Sigma_{jk}\|_{L^4}^4\Big)^{1/4} \bigg\} \\
& \lesssim n^{-1/2},
\end{align*}
where we have used $ \|X_{1j}\|_{L^8}\lesssim 1$, which can be established using an argument similar to~\eqref{eqn: rosendeltaj}.
Combining the last several steps shows that $\sum_{j\neq k}\E(\tilde\Sigma_{jk}^4)=o(n^{1/2})$. Applying this to~\eqref{eqn:offSigmat4} completes the proof.

\qed

\begin{lemma}
\label{lem:l4diagD}
If Assumption \ref{A:model}  holds,  then as $n\to\infty$, 
\begin{align}
\label{eqn:l4diagD}
\frac{\|\textup{Diag}(\Sigma^{1/2})\|_4^4-p}{\sqrt n} \rightarrow 0,
\end{align}
and
\begin{align}\label{eqn:l4diagDtilde}
\frac{\|\textup{Diag}(\tilde{\mathfrak{S}}^{1/2})\|_4^4-p}{\sqrt n}\xrightarrow{\P} 0.
\end{align}
In addition, the last limit holds when $\tilde{\mathfrak{S}}$ is replaced with $\hat{\mathfrak{S}}$.
\end{lemma}

\noindent \textit{Proof.} We first prove the limit~\eqref{eqn:l4diagD}. Recalling the notation in~\eqref{eqn:Deltadef}, we have
\begin{align*}
\Big|\|\textup{Diag}(\Sigma^{1/2})\|_4^4 - p  \Big|
& \, = \, \Big| \sum_{j=1}^p \big((1+B_{jj})^4 -1\big) \Big| \\
& \lesssim  \|B\|_*\Big(  1+ \|B\|_{\op}+  \|B\|_{\op}^2 + \|B\|_{\op}^3\Big) \\
& \lesssim  \|A\|_*\\
&=o(\sqrt p), 
\end{align*}
where we have used Assumption \ref{A:model} and the bound $|\lambda_j(B)|\leq |\lambda_j(A)|$ for all $j=1,\dots,p$ from \eqref{eqn:deltajbound}.

To show \eqref{eqn:l4diagDtilde}, let $\tilde R\in\R^{p\times p}$ be as defined in the proof of Lemma~\ref{lem:offdiag4thpowers} so that  $(\tilde {\mathfrak{S}}^{1/2})_{jj} = (\tilde\Sigma_{jj})^{1/2} + \tilde R_{jj}$ for all $j=1,\dots,p$, and $\|\tilde R\|_{\textup{op}}=o_{\P}(p^{-1/2})$. Also, by Lemma~\ref{lem:cov op bounds}, we have $\|\tilde\Sigma\|_{\textup{op}}=\mathcal{O}_{\P}(1)$, which gives
\begin{equation*}
\begin{aligned}
\|\textup{Diag}(\tilde{\mathfrak{S}}^{1/2})\|_4^4 & \, = \, \sum_{j=1}^p \tilde\Sigma_{jj}^2 \ + \  \sum_{l=0}^3 \binom{4}{l} \sum_{j=1}^p (\tilde\Sigma_{jj})^{l/2}\tilde R_{jj}^{(4-l)}\\
&=\, \sum_{j=1}^p \tilde \Sigma_{jj}^2 \ + \ o_{\P}(\sqrt p).
\end{aligned}
\end{equation*}
Furthermore, Lemma \ref{lem:sumhatsjj2} and Assumption~\ref{A:model} give
\begin{align*}
\sum_{j=1}^p \tilde\Sigma_{jj}^2 &= \sum_{j=1}^p \Sigma_{jj}^2 + \mathcal{O}_{\P}(1)\\
&=\sum_{j=1}^p (1+A_{jj})^2 +\mathcal{O}_{\P}(1),\\
&=p+\mathcal{O}\big(\|A\|_*(1+\|A\|_{\textup{op}})\big)+\mathcal{O}_{\P}(1)\\
&= p + o_{\P}(\sqrt{p}),
\end{align*}
which completes the proof.

\qed

\begin{lemma}
\label{lem:1-t order} If Assumption \ref{A:model}  holds, then as $n\to\infty$,
\begin{align*}
\tilde s^2 =o_{\P}(p^{-1/2}).
\end{align*}

\end{lemma}

\noindent \textit{Proof.}
The definition of $\tilde s$ in \eqref{eq:defs} gives
\begin{align}
\label{eqn:1minust}
 \tilde s^2 \leq  \bigg|\frac{\sum_{i\neq j} \tilde\Sigma_{ij}^2 -\frac{1}{n/2}\tr(\tilde \Sigma)^2 }{\sum_{i\neq j} \tilde\Sigma_{ij}^2}\bigg|.
\end{align}
Using Lemmas~\ref{lem:sumhatsjj2} and~\ref{lem:hatSigmaerror}, the numerator on the right side of~\eqref{eqn:1minust} satisfies
\begin{equation*}
\begin{split}
  \sum_{i\neq j} \tilde\Sigma_{ij}^2 -\ts\frac{1}{n/2}\tr(\tilde\Sigma)^2  & \ = \ \Big(\|\tilde\Sigma\|_F^2-\ts\frac{1}{n/2}\tr(\tilde\Sigma)^2\Big) - \displaystyle\sum_{j=1}^p \tilde\Sigma_{jj}^2\\
& \ = \ \Big(1 + \mathcal{O}_{\P}(\ts\frac{1}{n})\Big)\|\Sigma\|_F^2 \ - \ \sum_{j=1}^p \Sigma_{jj}^2 +\mathcal{O}_{\P}(1)\\[0.2cm]
&= \Big(\sum_{i\neq j} \Sigma_{ij}^2\Big) +\mathcal{O}_{\P}(\ts\frac{1}{n})\|\Sigma\|_F^2  +\mathcal{O}_{\P}(1)\\
&=\sum_{i\neq j} A_{ij}^2+ \mathcal{O}_{\P}(1),
\end{split}
\end{equation*}
where the notation $\Sigma=I+A$ defined in \eqref{eqn:Deltadef} has been used in the last step, along with Assumption~\ref{A:model}. Similarly, the denominator on the right side of~\eqref{eqn:1minust} satisfies
\begin{align*}
  \sum_{i\neq j} \tilde\Sigma_{ij}^2   & \ = \ \sum_{i\neq j} A_{ij}^2\ + \  \mathcal{O}_{\P}(1) + \ \ts\frac{1}{n/2}\tr(\tilde\Sigma)^2.
\end{align*}
Also, it is straightforward to show that $\tr(\tilde\Sigma)/\tr(\Sigma)=1+o_{\P}(1)$ because $\tr(\tilde\Sigma)$ is unbiased for $\tr(\Sigma)$ and $\var(\tr(\tilde\Sigma))=\frac{1}{n}\var(\|X_1\|_2^2)\lesssim 1$ by~\eqref{eqn:varorder}. Finally, since  we have  $\tr(\Sigma)\gtrsim p$ and $\|A\|_F^2\leq \|A\|_{\textup{op}}\|A\|_*=o(\sqrt p)$ under Assumption~\ref{A:model}, the proof is complete.

\qed

\begin{lemma}
\label{lem:hatSigmaerror}
If Assumption \ref{A:model} holds, then 
\begin{align*}
\|\hat\Sigma\|_F^2 - \ts\frac{1}{n}\tr(\hat\Sigma)^2 = \Big(1 + \mathcal{O}_{\P}(\ts\frac{1}{n})\Big)\|\Sigma\|_F^2.
\end{align*}
In addition, the same statement holds when $\hat\Sigma$ is replaced with $\tilde\Sigma$.
\end{lemma}

\noindent \textit{Proof.} First observe that
\begin{align*}
\|\hat\Sigma\|_F^2 & = \frac{1}{ n^2} \sum_{i=1}^n \sum_{j=1}^n (X_i^{\top} X_j)^2 \\
\frac{1}{n}\tr(\hat\Sigma)^2  & = \frac{1}{n^3} \sum_{i=1}^n \sum_{j=1}^n (X_i^{\top} X_i)(X_j^{\top} X_j). 
\end{align*}
Consequently, it can be checked that
\begin{align}\label{eqn:twoparts}
    \|\hat\Sigma\|_F^2 - \frac{1}{n}\tr(\hat\Sigma)^2 \ = \ \frac{1}{ n^2} \sum_{i\neq j} (X_i^{\top} X_j)^2   \ + \  \frac{1}{ n^3} \sum_{i > j}(X_i^{\top}X_i - X_j^{\top} X_j)^2.
\end{align}
The second sum on the right side is non-negative, and so we may show that it is $\mathcal{O}_{\P}(\frac{1}{n}\|\Sigma\|_F^2)$ by showing that its expectation is $\mathcal{O}(\frac{1}{n}\|\Sigma\|_F^2)$. In particular, we have
\begin{align*}
\frac{1}{ n^3} \sum_{i > j}\E\Big((X_i^{\top}X_i - X_j^{\top} X_j)^2\Big) & = \frac{1}{ n^3} \sum_{i > j}2 \var(X_i^{\top} X_i)\\
& \lesssim \frac{\|\Sigma\|_F^2}{n},
\end{align*}
where  \eqref{eqn:varorder} has been used in the last step.

Now we turn to the first sum on the right side of~\eqref{eqn:twoparts}. If this sum is viewed as an estimate of $\|\Sigma\|_F^2$, then to complete the proof, it is enough to show that the bias and standard deviation are both $\mathcal{O}(\frac{1}{n}\|\Sigma\|_F^2)$. A straightforward calculation shows that its mean is given by
$$\frac{1}{n^2}\sum_{i\neq j}\E((X_i\ttop X_j)^2)=(1-\ts\frac{1}{n})\|\Sigma\|_F^2 $$
and hence the bias is $\mathcal{O}(\frac{1}{n}\|\Sigma\|_F^2)$. 
Also, since $\frac{1}{n^2}\sum_{i\neq j}(X_i\ttop X_j)^2$ is proportional to a U statistic, the classical formula for the variance of a U statistic~\citep[p.163]{van2000asymptotic} gives
\begin{align*}
    \var\bigg(\frac{1}{n^2}\sum_{i\neq j}(X_i\ttop X_j)^2\bigg) &  \lesssim    \frac{1}{n}\var\big(\E((X_1^{\top} X_2)^2|X_1)\big) \, +\, \frac{1}{n^2}\var((X_1^{\top} X_2)^2)\\
    &\lesssim \frac{1}{n}\var(X_1\ttop \Sigma X_1)+\frac{1}{n^2}\E((X_1\ttop X_2)^4)\\[0.2cm]
    &\lesssim  \frac{\|\Sigma^2\|_F^2}{n}+ \frac{\|\Sigma\|_F^4}{n^2}\\
    &\lesssim 1,
\end{align*}
where Lemmas~\ref{lemma: EZ1MZ2^4} and \ref{lemma:covZMZ} have been used in the third step, and Assumption~\ref{A:model} has been used in the fourth step. Likewise, since Assumption~\ref{A:model} implies $\frac{1}{n}\|\Sigma\|_F^2\gtrsim 1$, it follows that the standard deviation of $\frac{1}{n^2}\sum_{i\neq j}(X_i\ttop X_j)^2$ is indeed $\mathcal{O}(\frac{1}{n}\|\Sigma\|_F^2)$.

\qed

\begin{lemma} 
\label{lem:cov op bounds} If Assumption~\ref{A:model} holds, then as $n\to\infty$,
\begin{align*}
\|\hat\Sigma\|_{\textup{op}} =\mathcal{O}_{\P}(1).
\end{align*}
In addition, the same result holds when $\hat\Sigma$ is replaced with $\tilde\Sigma$.
\end{lemma}

\noindent \textit{Proof.} 
Due to Assumption~\ref{A:model}, we have 
$$\hat\Sigma =   \Sigma^{1/2} \Big(\frac{1}{n}\sum_{i=1}^n Z_i Z_i^{\top}\Big) \Sigma^{1/2}.$$
It is known from~\citep[Theorem 3.1]{yin1988limit}  that the bound
\begin{align*}
 \Big\|\frac{1}{n} \sum_{i=1}^n Z_i Z_i^{\top}\Big\|_{\textup{op}}=\mathcal{O}_{\P}(1)
\end{align*}
holds under Assumption~\ref{A:model}, and so the lemma follows from $\|\Sigma\|_{\textup{op}}\lesssim 1$ and
the submultiplicative property of the operator norm.

\qed

\begin{lemma}
\label{lem:phatSS}
If Assumption \ref{A:model} holds, then for any fixed $\varepsilon\in(0,1)$, we have 
\begin{align}
\label{eqn:sumprobbound}
\sum_{j=1}^p \P\Big(\Big|\ts\frac{\hat\Sigma_{jj}}{\Sigma_{jj}} - 1\Big| > n^{\frac{-1+\varepsilon}{4}} \Big) \, \lesssim \, n^{-\varepsilon}.
\end{align}
Furthermore, there is a constant $c>0$ not depending on $n$ such that
\begin{equation}\label{eqn:mincovbound}
 \min_{1\leq j\leq p} \hat\Sigma_{jj} \, \geq c - o_{\P}(1).
\end{equation}
Lastly, both of the previous statements hold when $\hat\Sigma$ is replaced with $\tilde\Sigma$.
\end{lemma}

\noindent \textit{Proof.} The bound~\eqref{eqn:mincovbound} holds  because~\eqref{eqn:sumprobbound} implies 
$$\max_{1\leq j\leq p}\Big|\ts\frac{\hat\Sigma_{jj}}{\Sigma_{jj}}-1\Big|=o_{\P}(1),$$ 
and because Assumption~\ref{A:model}  implies $\min_{1\leq j\leq p}\Sigma_{jj}\gtrsim 1$. To prove~\eqref{eqn:sumprobbound}, we may assume without loss of generality that $\Sigma_{jj}=1$ for $j = 1,\ldots,p$. Using Chebyshev's inequality, we have
\begin{align*}
\P(|\hat\Sigma_{jj} - 1| > n^{\frac{-1+\varepsilon}{4}}) \leq n^{1-\varepsilon} \bigg\| \frac{1}{n} \sum_{i=1}^{n} (X_{ij}^2-1)  \bigg\|_{L^4}^4,
\end{align*}
and the right side can be bounded using Rosenthal's inequality (Lemma \ref{lem:Rosenthal}) to obtain
\begin{align*}
\P(|\hat \Sigma_{jj} - 1| > n^{\frac{-1+\varepsilon}{4}}) \lesssim n^{1-\varepsilon} \max\Big\{ \big(\ts\frac{1}{n}\var(X_{1j}^2)\big)^2, \big(\ts\frac{1}{n}\big)^3 \|X_{1j}^2-1\|_{L^4}^4 \Big\}.
\end{align*}
Since $X_{1j}^2 = Z_{1}^{\top} \Sigma^{1/2} e_j e_j^{\top} \Sigma^{1/2} Z_1$, it follows from Lemmas~\ref{lemma:covZMZ} and \ref{lem:Bai,quadra} as well as  the condition $\Sigma_{jj}=1$ that 
\begin{align*}
 \var(X_{1j}^2) & \lesssim \|\Sigma^{1/2} e_j e_j^{\top} \Sigma^{1/2}\|_F^2 = 1,\\
 \|X_{1j}^2 - 1\|_{L^4}^4 & \lesssim \|\Sigma^{1/2} e_j e_j^{\top} \Sigma^{1/2}\|_F^4=1.
\end{align*}
Hence, we obtain
\begin{align*}
\sum_{j=1}^p \P(|\hat \Sigma_{jj} - 1| > n^{\frac{-1+\varepsilon}{4}}) & \lesssim p n^{-1-\varepsilon},
\end{align*}
which leads to the stated result.

\qed

\subsection{Ratio consistency of the variance estimate}
\label{sec:sigma,cons}

\begin{proposition}
\label{prop:ratioconsit}
If Assumption \ref{A:model}  holds, then as $n \rightarrow \infty$,
\begin{align*}
\frac{\hat \sigma_n^2}{\sigma_n^2} \overset{\P}{\rightarrow} 1.
\end{align*}    
\end{proposition}

\noindent \textit{Proof.} From the definition of $\hat\sigma_n^2$ in~\eqref{eqn:hatsigmadef}, we have
\begin{equation}\label{eqn:varconsistency}
 \frac{\hat\sigma_n^2}{\sigma_n^2}= \frac{4(\hat{\var}(\|X_1\|_2^2))^2}{n\sigma_n^2(\sum_{j=1}^p \hat \Sigma_{jj}^2)^2} + \frac{ 2 \hat \var(\|X_1\|_4^4)}{n\sigma_n^2 {\|\hat{\mathfrak{S}}^{1/2}\|_4^8}}.
 \end{equation}
Since we know $\sigma_n^2\asymp 1/n$ from~\eqref{eqn:sigmaorder}, and $\|\hat{\mathfrak{S}}^{1/2}\|_4^4\geq p+o_{\P}(\sqrt n)$ from Lemma~\ref{lem:l4diagD}, we can show that the second term on the right side of~\eqref{eqn:varconsistency} is negligible by checking that $\hat{\var}(\|X_1\|_4^4)=o_{\P}(p^2)$. Indeed, since $\hat{\var}(\|X_1\|_4^4)$ is a non-negative random variable, the negligibility will follow if we can check that $\E(\hat{\textup{var}}(\|X_1\|_4^4))=o(p^2)$. To this end, the unbiasedness of $\hat{\var}(\|X_1\|_4^4)$  and Lemma~\ref{lem:centralXl4} imply $\E(\hat{\textup{var}}(\|X_1\|_4^4))=\var(\|X_1\|_4^4)=o( p^{3/2})$.

It remains to show that the first term on the right side of~\eqref{eqn:varconsistency} converges to 1 in probability. From the definition of~$\sigma_n^2$ in~\eqref{eqn:sigmandef} this term may be expressed as
\begin{equation*}
    \frac{4(\hat{\var}(\|X_1\|_2^2))^2}{n\sigma_n^2(\sum_{j=1}^p \hat \Sigma_{jj}^2)^2} = \bigg(\frac{\hat{\var}(\|X_1\|_2^2)}{\var(\|X_1\|_2^2)}\bigg)^2\bigg(\frac{\sum_{j=1}^p\Sigma_{jj}^2}{\sum_{j=1}^p\hat\Sigma_{jj}^2}\bigg)^2.
\end{equation*}
Lemma~\ref{lem:sumhatsjj2} shows that the second factor on the right is $1+o_{\P}(1)$, while Lemma S3 in~\citep{lopes2019bootstrapping} shows that the first factor on the right is $1+o_{\P}(1)$.

\qed

\section{Background results}\label{sec:app,background}

\begin{lemma}[Rosenthal's inequality~\citep{Rosenthal}]
\label{lem:Rosenthal}
Let $q \geq 2$, and let $\xi_1, \ldots, \xi_n$ be independent centered random variables. Then, there is an absolute constant $c>0$ such that 
$$
\Big\|\sum_{i=1}^n \xi_i \Big\|_{L^q} 
\leq c\cdot q \cdot \max \bigg\{\Big\|\sum_{i=1}^n \xi_i\Big\|_{L^2},\Big(\sum_{i=1}^n \big\|\xi_i \big\|_{L^q}^q\Big)^{1 / q}\bigg\}.
$$
\end{lemma}

\begin{lemma}[\cite{bhansali2007convergence}, Theorem 2.1]
\label{lem: quadra normal} Let $Y_1, Y_2,\dots$ be i.i.d.~random variables with $\E(Y_1)=0$, $\var(Y_1)=1$, and $\E(Y_1^4)<\infty$. For each integer $p=1, 2,\dots$, let $M^{(p)}$ be a real symmetric $p\times p$ matrix, and let $Q_p=\sum_{1\leq i,j\leq p} M_{ij}^{(p)}Y_iY_j$. Under these conditions, if  $\|M^{(p)}\|_F/\|M^{(p)}\|_{\textup{op}}\to \infty$ as $p\to\infty$, then the following limit holds as $p\to\infty$,
\begin{align*}
\frac{Q_p - \E(Q_p)}{\sqrt{\var(Q_p)}} \xrightarrow{\mathcal{L}} N(0,1).
\end{align*}
\end{lemma}

\begin{lemma}
\label{lemma:covZMZ}
Let $Z_1=(Z_{11},\dots,Z_{1p})$ be a random vector with i.i.d.~entries satisfying $\E(Z_{11})=0$, $\var(Z_{11})=1$, and $\E(Z_{11}^4)<\infty$.
 If $A,B\in \R^{p\times p}$ are symmetric matrices, then 
\begin{align}\label{eqn:covZMZ}
\cov\big(Z_1^{\top} A Z_1, Z_1^{\top} B Z_1\big) & = 2\tr(AB) + \big(\E(Z_{11}^4)-3\big) \sum_{j=1}^p A_{jj} B_{jj}.
\end{align}
In addition, if $X_1=\Sigma^{1/2}Z_1$ with $\Sigma\in\R^{p\times p}$ being a fixed non-zero positive semidefinite matrix, then
\begin{equation}\label{eqn:constraint}
          \frac{\var(\|X_1\|_2^2)-2\|\Sigma\|_F^2}{\sum_{j=1}^p\Sigma_{jj}^2}
  =\frac{\E(\|X_1\|_4^4)-3\sum_{j=1}^p\Sigma_{jj}^2}{\|\Sigma^{1/2}\|_4^4}.
\end{equation}
\end{lemma}

\noindent \textit{Proof.} The identity~\eqref{eqn:covZMZ} is given in equation~(9.8.6) of~\cite{bai2010spectral}. To show~\eqref{eqn:constraint}, if we take $A=B=\Sigma$, then
\begin{equation*}
  \var(\|X_1\|_2^2) = 2\|\Sigma\|_F^2 + (\E(Z_{11}^4)-3)\sum_{j=1}^p \Sigma_{jj}^2. 
\end{equation*}
Meanwhile, if we take $A=B=\Sigma^{1/2}e_ke_k\ttop \Sigma^{1/2}$ with $e_k\in\R^p$  denoting the $k^{th}$ standard basis vector, then summing over $k=1,\dots,p$ gives
\begin{equation*}
    \sum_{k=1}^p\big(\E(X_{1k}^4)-3\Sigma_{kk}^2\big) = \big(\E(Z_{11}^4)-3\big) \sum_{k=1}^p \sum_{j=1}^p(\Sigma^{1/2})_{jk}^4.
\end{equation*}
Finally, eliminating $\E(Z_{11}^4)-3$ from the previous two equations leads to the stated result.

\qed

\begin{lemma}[\cite{bai2010spectral}, Lemma B.26]
\label{lem:Bai,quadra} Let $A\in\R^{p\times p}$ and $q\geq 1$ be fixed. If $Z_1=(Z_{11},\dots,Z_{1p})$ is a random vector with i.i.d.~entries satisfying $\E(Z_{11})=0$ and $\var(Z_{11})=1$, then
\begin{align*}
\E\Big(\big|Z_1^{\top} A Z_1 -\tr(A) \big|^q\Big) \leq C_q\Big( (\E(Z_{11}^4))^{q/2}\, \|A\|_{F}^q + \E(|Z_{11}|^{2q})\,\tr((AA\ttop)^{q/2})\Big),
\end{align*}
where $C_q>0$ is a number depending only on $q$.
\end{lemma}

\end{document}